\documentclass[11pt,twoside]{amsart}
\usepackage{amsmath,amscd,amsthm,amssymb,amsxtra,latexsym,epsfig,epic,graphics,mathdots,amssymb,enumitem}
\usepackage{stmaryrd}
\usepackage{MnSymbol}
\usepackage[all]{xy}
\usepackage[utf8]{inputenc}
\usepackage[usenames,dvipsnames]{xcolor}
\usepackage{tcolorbox}
\tcbuselibrary{skins,listings,theorems}
\usepackage{hyperref}
\newcommand{\inter}[1]{\llbracket {#1}\rrbracket}
\usepackage{eucal}
\usepackage[sans]{dsfont}
\usepackage{graphicx,color}

\newtheorem{MainTheorem}{Theorem}

\newtheorem{MainCorollary}[MainTheorem]{Corollary}

\newtheorem{Theorem}{Theorem}[section]
\newtheorem{Lemma}[Theorem]{Lemma}
\newtheorem{Proposition}[Theorem]{Proposition}

\theoremstyle{definition}
\newtheorem{Definition}[Theorem]{Definition}

\newtheorem{Remark}[Theorem]{Remark}

\newtheorem{Example}[Theorem]{Example}

\newcommand{\ZZ}{{\mathds Z}}
\newcommand{\NN}{{\mathds N}}
\newcommand{\MM}{M}

\newcommand{\PP}{{\mathds P}}

\newcommand{\VV}{{\mathds V}}
\newcommand{\AAA}{{\mathbf{A}}}
\newcommand{\III}{{\mathds 1}}

\newcommand{\cE}{{\mathcal E}}

\newcommand{\cN}{{\mathcal N}}
\newcommand{\cL}{{\mathcal L}}
\newcommand{\cQ}{{\mathcal Q}}
\newcommand{\cK}{{\mathcal K}}
\newcommand{\cF}{{\mathcal F}}
\newcommand{\cI}{{\mathcal I}}
\newcommand{\cJ}{{\mathcal J}}

\newcommand{\QQ}{{\mathcal Q}}

\newcommand{\cT}{{\mathcal T}}

\newcommand{\cO}{{\mathcal O}}
\newcommand{\cV}{{\mathcal V}}
\newcommand{\rH}{{\mathrm H}}
\newcommand{\rM}{{\mathrm M}}
\newcommand{\rh}{{\mathrm h}}

\newcommand{\ssl}{{\mathfrak {sl}}}
\newcommand{\fe}{{\mathfrak {e}}}
\newcommand{\fm}{{\mathfrak {m}}}
\newcommand{\fg}{{\mathfrak {g}}}
\newcommand{\fh}{{\mathfrak {h}}}
\newcommand{\fl}{{\mathfrak {l}}}
\newcommand{\fM}{{\mathfrak {M}}}
\newcommand{\fH}{{\mathfrak {H}}}
\newcommand{\fD}{{\mathfrak {D}}}

\newcommand{\sat}{{\mathrm{sat}}}
\newcommand{\sing}{{\mathrm{sing}}}
\newcommand{\cha}{{\mathrm{char}}}
\newcommand{\supp}{{\mathrm{supp}}}
\newcommand{\SL}{\mathrm{SL}}
\newcommand{\GL}{\mathrm{GL}}
\newcommand{\PGL}{\mathrm{PGL}}
\newcommand{\diag}{\mathrm{diag}}
\newcommand{\kk}{{\mathds k}}
\newcommand{\Pic}{\operatorname{Pic}}
\newcommand{\Cl}{\operatorname{Cl}}
\newcommand{\rk}{\mathrm{rk}}
\newcommand{\gr}{\operatorname{gr}}
\newcommand{\trace}{\operatorname{tr}}
\newcommand{\Hom}{\operatorname{Hom}}
\newcommand{\cHom}{\mathcal{H}om}
\newcommand{\End}{\operatorname{End}}
\newcommand{\Ext}{\operatorname{Ext}}

\newcommand{\coker}{\operatorname{coker}}

\newcommand{\Proj}{\operatorname{Proj}}

\newcommand{\ra}{{\boldsymbol{\mathrm{a}}}}
\newcommand{\rb}{{\boldsymbol{\mathrm{b}}}}
\newcommand{\bg}{\boldsymbol{\mathrm{g}}}
\newcommand{\bh}{\boldsymbol{\mathrm{h}}}
\newcommand{\rdet}{\mathbf{det}}
\newcommand{\rsyz}{\mathbf{syz}}
\newcommand{\udet}{\underline{{\mathbf{det}}}}
\newcommand{\rf}{{\boldsymbol{\mathrm{f}}}}
\newcommand{\ri}{{\boldsymbol{\mathrm{i}}}}
\newcommand{\la}{\operatorname{\leftarrow}}
\newcommand{\id}{\operatorname{id}}

\newcommand{\tra}{{\operatorname{t}}}

\begin{document}
\sloppy

\title[On stability of logarithmic tangent sheaves]{On stability of
  logarithmic tangent sheaves. \\ Symmetric and generic determinants}

\author[D. Faenzi, S. Marchesi]{Daniele Faenzi, Simone Marchesi}
\thanks{D.F.
  partially supported by ISITE-BFC ANR-lS-IDEX-OOOB,
  ANR Fano-HK ANR-20-CE40-0023,
 CAPES-COFECUB project Ma 926/19 and EIPHI ANR-17-EURE-0002}

\address{Daniele Faenzi,
  Institut de Mathématiques de Bourgogne,
  Universit\'e de Bourgogne et Franche-Comté,
  9 avenue Alain Savary -- BP 47870,
  21078 Dijon Cedex, France.
} \email{daniele.faenzi@u-bourgogne.fr}

\address{Simone Marchesi,
Facultat de Matemàtiques i Informàtica, Universitat de Barcelona, Gran
Via de les Corts Catalanes 585 -- 08007 Barcelona, Spain.
} \email{marchesi@ub.edu}



\begin{abstract}
We prove stability of logarithmic tangent sheaves of singular
hypersurfaces $D$ of the projective space with constraints on the
dimension and degree of the singularities of $D$.
As main application, we prove that
determinants and symmetric determinants have simple (in
characteristic zero, stable) logarithmic
tangent sheaves
and we describe an open dense piece of the associated moduli space.
\end{abstract}

\maketitle

\section*{Introduction}

Given a hypersurface $D \subset \PP^N$ defined by a homogeneous form
$F$ of degree $d$ over a field $\kk$, the vector fields on
$\PP^N$ which are tangent to $D$ define the logarithmic tangent sheaf $\cT_D$
of $D$. This sheaf is the first syzygy of the Jacobian ideal sheaf $\cJ_D$ of $D$,
as the partial derivatives $\nabla (F)$ of $F$, that is, the generators of the Jacobian ideal $J_D$,
express it as the kernel of the
Jacobian matrix of $F$:
\begin{equation}\label{seq-def1}
0 \rightarrow \cT_D \rightarrow (N+1).\cO_{\PP^N} \xrightarrow{\nabla (F)}
\cO_{\PP^N}(d-1).
\end{equation}

If the characteristic of $\kk$ does not divide
$d$, the sheaf $\cT_D(1)$ is a subsheaf of $\cT_{\PP^N}$, usually
denoted by $\cT_{\PP^N}\langle D \rangle$, and the quotient of
$\cT_{\PP^N}$ by $\cT_{\PP^N}\langle D \rangle$ is the equisingular normal sheaf of $D$.
The sheaf $\cT_{\PP^N}\langle D \rangle$, or rather its dual, often denoted by
$\Omega_{\PP^N}(\log D)$, was studied in \cite{deligne:equations} and
\cite{saito:logarithmic} in connection with Hodge theory.
All these sheaves play a major role in the
deformation theory of the embedding $D\hookrightarrow \PP^N$, see
\cite[Section 3.4]{sernesi:deformations}. 
The graded module of global sections of $\cT_D$, which we denote by
$T_D$, is called the module of logarithmic derivations, or of Jacobian
syzygies of $F$.
It has been also studied in detail, most notably for hyperplane arrangements, see for
instance \cite{orlik-terao:arrangements-hyperplanes}.

For some noteworthy classes of hypersurface singularities the
logarithmic tangent sheaf is locally free, and this plays an important
role in the theory of discriminants and unfolding 
of singularities, see for instance \cite{buchweitz-ebeling-von_bothmer}.
For some remarkable classes of divisor, the module $T_D$ is itself
free, see for instance \cite{terao},
so that $\cT_D$ splits as a direct sum of line bundles.

In contrast to this, for some interesting classes of hypersurfaces the
sheaf $\cT_D$ is stable. This happens for instance for generic
arrangements of at least $N+2$ hyperplanes
\cite{dolgachev-kapranov:arrangements}, but also for many highly
non-generic arrangements
see \cite{faenzi-matei-valles,abe-faenzi-valles}.
Recall that $\cT_D$ is stable when $D$ is smooth and, moreover, the stability of $\cT_D$ for hypersurfaces with isolated singularities was
studied in \cite{dimca:jacobian-syzygies}, in connection with the
Torelli problem, on whether $D$ can be reconstructed from $\cT_D$.
Stability of $\cT_D$ is a fundamental preliminary step to connect the
study of equisingular deformations of $D$ to moduli problems of
sheaves over $\PP^N$.
However, few methods for proving stability of $\cT_D$ are
available today, indeed very little seems to be
known besides the case of arrangements and 
isolated singularities of curves and surfaces.

\medskip

The goal of this paper is to propose some tools to prove
stability in a wide range of situations.
The general strategy is find a suitable
closed subvariety $X \subset \PP^N$ where we may prove that the
restriction of $\cT_D$ is stable and then argue that this implies
stability of $\cT_D$ itself.

The first possibility to explore is to take $X$ to be a linear space
disjoint from the singular locus $\sing(D)$ of $D$.
In the first part of this paper (see \S \ref{sec-lowdim}) we show that
$\cT_D|_X$ is 
stable provided some 
 vanishing of global sections of the reflexive
hulls of exterior powers of $\cT_D$
in terms of the codimension of singularities of $D$.
More specifically, setting $s=\dim(\sing(D))$ and assuming $s \le N-2$, we obtain the following result.

\begin{MainTheorem} \label{stabileH0}
  Assume that for all integers $p$ with $1 \le p \le s+1$ we have:
  \[
    \rH^0\left(\wedge^p \cT_D(q)^{**} \right)=0,
      \qquad \mbox{with:} \qquad
      q=\left\lfloor {\frac{(d-1)p}N} \right\rfloor.
  \]
  Then $\cT_D$ is slope-stable.
\end{MainTheorem}

We may also formulate the result in terms of the Hilbert function of
$\cT_D$ only.

\begin{MainCorollary} \label{corstabileH0}
  The sheaf $\cT_D$ is slope-stable if:
  \[
    \rH^0\left(\cT_D(q) \right)=0,
      \qquad \mbox{with:} \qquad
      q=\left\lfloor {\frac{(d-1)(s+1)}N} \right\rfloor.
  \]
\end{MainCorollary}

For isolated hypersurface singularities, this allows to generalize
\cite[Theorem 1.3]{dimca:jacobian-syzygies} and
\cite[Theorem 3.3]{dimca:versality-bounds} to arbitrary dimension. 

\begin{MainTheorem} \label{gen-dimca}
  Assume $\dim(\sing(D))=0$ and set $q=\left\lfloor {\frac{d-1}N} \right\rfloor$. Then $\cT_D$ is stable if:
  \[
    \deg(\sing(D)) < (d-q-1)(d-1)^{N-1}.
  \]
\end{MainTheorem}

In the second part of this paper, we consider some natural families
of divisors, not covered by the previous results, where stability of
$\cT_D$ can be proved.
Indeed, many interesting hypersurfaces tend to have singularities
of small codimension, for instance many divisors coming from orbit
closures, discriminants or from moduli theory are highly singular.
In this case, one strategy we propose is to pick a subvariety $X
\subset \PP^N$ disjoint from $\sing(D)$, such that $\cT_D$ restricts over
$X$ to a vector bundle of some special form, whose stability is under
control, and deduce from this the stability of $\cT_D$.
A natural candidate for this is the bundle of principal parts $\cE_{n-1}$.
This is defined as kernel of the evaluation of sections $\rH^0(\cO_X(n-1))
\otimes \cO_X \to \cO_X(n-1)$.
We contend that in some relevant situations $\cT_D$ will restrict to a
slope-stable bundle of principal parts $\cE_{n-1}$ and that this suffices to
prove stability of $\cT_D$ itself.
Let us point out that vector bundles of principal parts on projective
spaces, and in particular their stability, is a matter of independent
interest, see for example \cite{maakestad:principal,re:principal}.
We contribute to this by showing that the vector
bundle of principal parts on a smooth quadric surface, namely
on $\PP^1\times \PP^1$, is slope-stable (see
Proposition \ref{En is stable}). For this we make use of
representations of a quiver supported on a planar graph, rather than the tree
appearing in \cite{re:principal}.
 For this last point we need the base field to be of
  characteristic zero as the relevant representation theory is more
  tricky in positive characteristic. However the sheaf $\cT_D$ is
  easily proved to be simple in any characteristic.

Going back to the main families of divisors where our strategy
applies, let us first mention symmetric discriminants, see \S
\ref{section:symmetric_determinants}.
In this case, we argue that the suitable subvariety $X$ is a
projective plane, where stability of vector bundles of principal parts
is well-known. Also, in view of the Goto–J\'{o}zefiak–Tachibana's
resolution, see \cite{jozefiak:symmetric,goto-tachibana} (see also
  \cite[Section 6.3.8]{weyman:tract}), we get that $\cT_D$ is a
  \textit{Steiner sheaf}, that is, it has a linear resolution of length
  two. Altogether, the result is the following.

\begin{MainTheorem} \label{main:symmetric_determinant}
Let $n\ge 2$ be an integer and let $D$ be the determinant divisor of symmetric $n\times
n$ matrices in $\PP^{{n+1 \choose 2}-1}$. Then the logarithmic sheaf $\cT_D$ satisfies:
\begin{enumerate}[label=\roman*)]
\item \label{thereso} the sheafified minimal graded free resolution of $\cT_D$ takes the form:
  \[ 
    0 \to {n \choose 2}.\cO_{\PP^N}(-2) \to (n^2-1).\cO_{\PP^N}(-1) \to \cT_D \to 0;
  \]
 \item \label{isprincipal} the restriction of $\cT_D$ to a generic
  plane $P \subset \PP^N$ is
  isomorphic to the
  bundle of principal parts $\cE_{n-1}$ defined as kernel of the evaluation map:
  \[
    \binom{n+1}{2}.\cO_P \to  \cO_P(n-1);
  \]

  \item \label{isstable} if $\cha(\kk)=0$, the logarithmic sheaf $\cT_D$ is slope-stable. 
  \end{enumerate}
\end{MainTheorem}

The next family we wish to mention is one of the main characters of
this paper, namely the \textit{generic determinant}, namely the divisor
$D$ defined as determinant of an $n\times n$ matrix of
variables $(x_{i,j})_{1 \le i,j \le n}$ in $\PP^N=\PP^{n^2-1}$.
This time, the suitable subvariety $X \subset \PP^N$ is a smooth
quadric surface. As we mentioned above, the bundle of principal parts
$\cE_{n-1}$  on $X$ is slope-stable, so the main point is to prove that
$\cT_D$ restricts over $X$ to the bundle $\cE_{n-1}$.
To do this, we analyze the Artinian reduction $A_L$ of the Jacobian
algebra of $D$ over the linear span $L \simeq \PP^3 \subset \PP^N$ of
$X$. In particular, we prove a quadratic Lefschetz property of
$A_L$, which in turn is obtained by specializing $L$ to a well-chosen
linear section which we call \textit{semigeneric}. It will turn out
that the intersection $D \cap L$ is a singular surface which is
resolved by a projective plane, blown-up at $n(n-1)$
complete intersection points. Studying carefully the divisors on this
blow-up we are able to prove
the next result for all $n \ge 2$.

\begin{MainTheorem} \label{main:determinant}
Let $n\ge 2$ be an integer and let $D$ be the generic determinant divisor of $n\times
n$ matrices in $\PP^{{n^2}-1}$. 
  Then the logarithmic sheaf $\cT_D$ satisfies:
  \begin{enumerate}[label=\roman*)]
\item  the sheafified minimal graded free resolution of $\cT_D$ takes
  the form:
  \[0 \to \cO_{\PP^N}(-n-1) \to n^2.\cO_{\PP^N}(-2) 
  \to 2(n^2-1).\cO_{\PP^N}(-1) \to \cT_D \to 0; 
    \]
  \item the restriction of $\cT_D$ to a generic
  quadric surface $X \simeq \PP^1 \times \PP^1 \subset \PP^N$ is
  isomorphic to the
  bundle of principal parts $\cE_{n-1}$ defined as kernel of the evaluation map:
  \[
    n^2.\cO_X \to  \cO_X(n-1);
  \]
\item the sheaf $\cE_{n-1}$ is simple and, if $\cha(\kk)=0$,  $\cE_{n-1}$ and $\cT_D$ are slope-stable.
\end{enumerate}
\end{MainTheorem}

In light of the last item, it is natural to investigate the
moduli space of simple sheaves -- or, if $\cha(\kk)=0$, of stable
sheaves -- which we denote by $\fM_n$, that contains
the sheaf $\cT_D$. This is done in \S
\ref{section:familiesdeterminants}. 
  Let us set up the framework needed to state our result in this
  direction. Let $n \ge 3$ be an integer.
  Consider two $n$-dimensional vector spaces $U$ and $V$ and
  the group $\SL(U) \times \SL(V)$. Put $\AAA=\Hom_\kk(V,U) \simeq
  V^* \otimes U$ and
  consider the standard representation of $\SL(U)$ on $U$, tensored
  with $\id_{V^*}$, so that $\SL(U)$ acts linearly on
  $\AAA$. This action commutes with the $\SL(V)$-action on
  $\AAA$ obtained via the dual representation on $V^*$ tensored with
  $\id_U$. So $\AAA$ is a representation of $\SL(U) \times \SL(V)$, which is faithful. We get an injective map $\SL(U) \times \SL(V) \to
  \GL(\AAA)$ and an induced injective morphism $\SL(U) \times \SL(V) \to
  \PGL(\AAA)$ which identifies $\SL(U) \times \SL(V)$ to a closed
  subgroup of $\PGL(\AAA)$. 
  
  For any $\rf \in \End_\kk(\AAA)$, we consider a map
  $M_\rf : U \otimes \cO_{\PP(\AAA)}(-1) \to V \otimes
  \cO_{\PP(\AAA)}$ canonically
  associated with $\rf$.
  When $\det(M_\rf) \ne 0$, it turns out that, setting $D_\rf=
  \VV(\det(M_\rf))$, we 
  have:
  \[
    \mbox{$\cT_{D_\rf}$ is simple} \qquad \Leftrightarrow \qquad [\rf]
    \in \PGL(\AAA) \qquad \Leftrightarrow \qquad  \mbox{$\cT_{D_\rf}$ is slope-stable (for $\cha(\kk)=0$).} 
  \]
  The subgroup $\SL(U) \times \SL(V)$ acts on the matrices $M_\rf$ by
  two-sided multiplication and this does not alter the isomorphism
  class of $\cT_{D_\rf}$. Hence,
  denoting by $\fM_n$ the moduli space of simple (or, in
  characteristic zero, stable) sheaves having the same Hilbert
  polynomial as $\cT_D$, we see that the assignment $\Psi : [\rf] \mapsto
  \cT_{D_\rf}$ defines a morphism: 
  \[
    \Psi : \PGL(\AAA)/\SL(U) \times \SL(V) \to \fM_n 
  \]

  Of course, the transpose ${}^\tra M_\rf$ of $M_\rf$ lands on the same divisor $D_\rf$.
  Our main result concerning $\fM_n$ is that, up to
  the $2:1$ cover arising from transposition, the 
  map $\Psi$ captures essentially the whole geometry of the open dense piece
  of $\fM_n$ consisting of logarithmic sheaves.
  \begin{MainTheorem} \label{main:determinant-moduli}
  The map $\Psi$ is an étale $2:1$ cover onto its image. The image of
  $\Psi$ is a smooth open affine piece of an
  irreducible component of $\fM_n$, of dimension $(n^2-1)^2$.
\end{MainTheorem}

  This is related to classical work of Frobenius,
  \cite[\S 7.1]{frobenius:gruppen}, equivalent to the fact that the
  determinant is a $2:1$ map from
  $\PGL(\AAA)/\SL(U) \times \SL(V)$ to the locus $\fD_n$ of
  determinantal hypersurfaces.
  The étale nature of this map is proved in
  \cite{reichstein-vistoli:determinantal}, see Remark (1) after
  Theorem 1.3 of that paper for a discussion of the interplay with
  Frobenius' theorem and related literature.
  So Theorem \ref{main:determinant-moduli} can be though of as an
  interpretation of 
  \cite{reichstein-vistoli:determinantal} in terms of moduli spaces of
  simple or stable sheaves.
  Our method is characteristic-free (except for stability of bundles of
  principal parts on $\PP^1 \times \PP^1$) and relies on a direct
  proof of the Torelli theorem (in the sense of Dolgachev-Kapranov)
  for determinantal hypersurfaces and the divisor class group of
  determinantal hypersurfaces.

An analogous description as an algebraic group quotient, that is, a
homogeneous space, could be obtained as
well for the case of hypersurfaces defined by determinants of
symmetric matrices. Nevertheless, recall that the generic element of
the moduli space of semistable Steiner sheaves is locally free,
therefore the image of such quotient would sit as a closed
subscheme of the relevant moduli space. 
So there is no direct analogue of Theorem
\ref{main:determinant-moduli} for symmetric determinants.

\bigskip

\noindent\textbf{Notation.} Let us fix some notation that will be used
throughout this paper. Denote by $\kk$ a field. The assumptions on
$\kk$ may
change in different sections. We use the polynomial ring
$R=\kk[x_0,\ldots,x_N]$ and, if $A$ is a graded $R$-module, we denote
by $A_p$ its degree-$p$ summand.

If $U$ is a $\kk$-vector space, we write $\PP(U)$ for the set of
hyperplanes of $U$. 
For an integer $m$, if $\cE$ is a vector space, or module, or a sheaf,
we write $m.\cE$ for the direct sum of $m$ copies of 
 $\cE$. Put $\PP^N=\PP((N+1).\kk)=\Proj(R)$.
Given a non-zero homogeneous polynomial $F \in R$ of degree $d$, write
$D =\VV(F)$ for the hypersurface of $\PP^N$
defined by $F$. Denoting by $\nabla(F)$ its Jacobian matrix, the
Jacobian ideal $J_D$ is the ideal generated by $\nabla(F)$ and $\cJ_D$
is the Jacobian ideal sheaf.
The \textit{logarithmic tangent sheaf} $\cT_D$ associated to $D$ is
defined as the kernel of the gradient of $F$:
\[
0 \rightarrow \cT_D \rightarrow (N+1).\cO_{\PP^N} \xrightarrow{\nabla (F)}
\cJ_D(d-1) \to 0. 
\]

Given a coherent sheaf $\cF$ and $i \in \NN$, we write
$\rH_*^i(\cF)$ for the $i$-th cohomology module of $\cF$, namely
$\rH_*^i(\cF)=\bigoplus_{t \in  \ZZ}\rH^i(\cF(t))$. 
The \textit{module of logarithmic derivations} of $D$ is defined as
$T_D=\rH^0_*(\cT_D)$.
Moreover, $\sing(D)$ will denote the singular
locus of $D$, equipped with its natural scheme structure, which is to say $\sing(D)=\VV(J_D)$.
We write $s =\dim(\sing(D))$.

We will say that a coherent sheaf on a subvariety $X \subset \PP^N$
is stable or semistable if it is so in the sense of Gieseker, with
respect to the hyperplane divisor on $X$.
We will use the notion of slope-stability, again with respect to the
hyperplane divisor, and use that slope-stability implies stability
while semistability implies slope-semistability. We refer to \cite{huybrechts-lehn:moduli}
for basic material on semistability of sheaves.

\noindent\textbf{Acknowledgements.} We would like to thank the
Institute of Mathematics, Statistics and Scientific Computing in
Campinas and the Institut de Math\'ematiques de Bourgogne in Dijon for the warm hospitality and for providing the best
working conditions. We are grateful to Ronan Terpereau
for useful comments.
 We wish to thank the referee for suggesting major improvements to
the paper and pointing out the reference \cite{reichstein-vistoli:determinantal}.

\section{Stability for low dimensional singularities}\label{sec-lowdim}

\label{section:hoppe}

In this section, we study the general case, that is hypersurfaces $D$
inside $\PP^N$ of degree $d \ge 2$.
Specifically, \S \ref{section:proofs} is devoted to prove
Theorems \ref{stabileH0} and  \ref{gen-dimca} and Corollary
\ref{corstabileH0}.
Sharpness of these results is also discussed briefly.
Finally, in \S \ref{section:torelli} we provide some further remarks and
comments about the Torelli problem for logarithmic derivations, namely
the question of whether the hypersurface $D$ can be reconstructed from
the sheaf $\cT_D$. In this section $\kk$ is an arbitrary field.

\subsection{Stability of sheaves of logarithmic derivations}
 \label{section:proofs} 

 We first prove Theorem \ref{stabileH0}. Our strategy is to exclude the
existence of destabilizing subsheaves having rank up to $s+1$ making
use of the vanishing assumptions of spaces of global sections of
reflexive hulls of exterior powers with a refinement of Hoppe's criterion.

Then, we take care of potentially destabilizing subsheaves of rank
between $s+2$ and $n-1$ by restricting $\cT_D$ to a linear space $L$
of codimension $s+1$
disjoint from the singular locus $\sing(D)$ and working on the resulting
Koszul complex of $\nabla (F)|_L$.

\begin{proof}[Proof of Theorem \ref{stabileH0}] 
We assume that $\cT_D$ is unstable despite satisfying the assumptions
and we seek a contradiction. 
Without loss of generality, we may assume that the field $\kk$ is
algebraically closed.

Consider a destabilizing subsheaf $\cK$ of $\cT_D$, set $r=\rk(\cK)$ and
put $c=\deg(c_1(\cK))$ so that $r < N$ and:
\begin{equation}
  \label{slope-inequality}
  \frac c r \ge \frac{1-d}{N}.  
\end{equation}

Without loss of generality, we may assume that $\cT_D/\cK$ is
torsion-free.
The embedding $j : \cK \hookrightarrow \cT_D$ gives a non-trivial map $
\wedge^r \cK \to \wedge^r \cT_D$ and, applying the bi-duality functor, we
get a non-trivial map:
\[
  j_r : \cO_{\PP^N}(c) \simeq (\wedge^r \cK)^{**} \to (\wedge^r \cT_D)^{**}.
\]
The image of $j_r$ is a quotient of $\cO_{\PP^N}(c)$, hence it is
 a torsion sheaf unless $j_c$ is injective. The former
case is excluded since this image sits in $(\wedge^r \cT_D)^{**}$, so
$j_r$ is injective and
\begin{equation}
  \label{section-exists}
  \rH^0(\PP^N,\wedge^r \cT_D(-c)^{**}) \ne 0.  
\end{equation}

This non-vanishing contradicts our vanishing assumptions if $r \le
s+1$. Hence we must have $s+2 \le
r \le N-1$. In other words, we have to prove that $\cT_D$ has no
destabilizing subsheaf of rank $r$ with $s+2 \le r \le N-1$.
So the proof is finished for $s=N-2$ but needs further argumentation for
$s < N-2$.

To comply with this, we
consider a linear subspace $L$ of $\PP^N$ of codimension $s+1$ which
is skew to $\sing(D)$ and meets transversely the locus
where $\cT_D/\cK$ is not locally free. Observe
that $\dim(L)=N-s-1 \ge 2$.
Denote $\cF = (\cT_D)|_L$. Then the exact sequence defining $\cT_D$ restricts to:
\begin{equation}
  \label{definitionF}
  0 \to \cF \to (N+1).\cO_L  \to \cO_L(d-1)  \to 0  
\end{equation}
and therefore the sheaf $\cF$ is locally free.
Moreover, the map $j$ restricts to an injective map $j_L : \cK|_L \to
\cF$ and, taking exterior powers of $j_L$, we get:
\[
  \rH^0(L,\wedge^r \cF(-c)) \ne 0.  
\]
Using the natural isomorphism $\wedge^r \cF \simeq \wedge^{N-r}
\cF^*(1-d)$, this amounts to:
\begin{equation}
  \label{section-exists-L}
  \rH^0(L,\wedge^{N-r} \cF^*(1-d-c)) \ne 0.  
\end{equation}

Now, since $d \ge 2$ and $r < N$, the inequality \eqref{slope-inequality} gives:
\[
  c \ge \frac{r(1-d)}{N} > 1-d,
\]
or, equivalently, $1-d-c<0$. Also, $s+2 \le r \le N-1$ gives $1 \le N-r \le
N-s-2 = \dim(L)-1$. Therefore, to reach the desired contradiction, it suffices to show:
  \begin{equation}
    \label{dadimo}
  \rH^0(L,\wedge^p \cF^*(-1))=0, \qquad \mbox{for all integers $p$ with
    $1\le p \le \dim(L)-1$}.
  \end{equation}

  To get this, we dualize \eqref{definitionF} and take $p$-th exterior power to
  write an exact complex:
    \begin{align*}
      \bigwedge^p\Big( \cO_L(1-d) \to (N+1).\cO_L\Big) \longrightarrow \wedge^p \cF^* \to 0.    
  \end{align*}

  Tensoring with $\cO_L(-1)$ and taking cohomology,
  since $p \le \dim(L)-1$ we get $\rH^0(L,\wedge^p \cF^*(-1))=0$. So
  \eqref{dadimo} is proved and the theorem as well.
\end{proof}

\begin{proof}[Proof of Corollary \ref{corstabileH0}]
Given an integer $p$ with $1 \le p \le s+1$, the $p$-th exterior power
of the injection $i : \cT_D \to (N+1).\cO_{\PP^N}$ gives maps:
\[
  \bigwedge^p \cT_D \xrightarrow{i_1} (N+1).\bigwedge^{p-1} \cT_D \to \cdots \xrightarrow{i_{p-1}} (N+1)^{(p-1)}.\cT_D,
\]
and taking reflexive hulls the composition $i_{p-1}\circ \cdots
\circ i_1$ gives:
\[
  \xymatrix{
    \bigwedge^p \cT_D \ar^-{i_{p-1}\circ \cdots \circ i_1}[r] \ar[d] & (N+1)^{(p-1)}.\cT_D \ar@{=}[d]\\
    \left(\bigwedge^p \cT_D\right)^{**} \ar[r] & (N+1)^{(p-1)}.\cT_D^{**} }
  \]
  Since the kernel of each of the maps $i_1,\ldots,i_{p-1}$ is a torsion sheaf,
  we get that $i_{p-1}\circ \cdots \circ i_1$ induces an injective
  map:
  \[
    \wedge^p \cT_D(q)^{**} \hookrightarrow (N+1)^{(p-1)}.\cT_D(q).
  \]
  Therefore, for all $p$ with $1 \le p \le s+1$, setting
  $q_p=\left\lfloor {\frac{(d-1)p}N} \right\rfloor$ and assuming
  $\rH^0(\cT_D(q))=0$ we get that
  $\rH^0(\wedge^p\cT_D(q_p)^{**})=0$ for all $p$ so Theorem \ref{stabileH0} gives
  stability of $\cT_D$.
\end{proof}

\begin{proof}[Proof of Theorem \ref{gen-dimca}]
  We assume $\deg(\sing(D)) < (d-q-1)(d-1)^{N-1}$ and prove that $\cT_D$ is slope-stable.
  Since the sheaf $\cT_D$ is reflexive and $\dim(\sing(D))=0$, 
  in view of Corollary \ref{corstabileH0}, we only have to check:
  \[
    \rH^0(\cT_D(q))=0.
  \]

  The degree (which is to say, the length) of the 0-dimensional subscheme $\sing(D) \subset \PP^N$ is the  
  total Tjurina number of $\sing(D)$, obtained as the sum of the length
  of the localization of $\sing(D)$ at the points of the set-theoretic support of $\sing(D)$.

  Consider the minimal degree relation of $J_D$, that is,
  the smallest integer $r$ such
  that $\rH^0(\cT_D(r)) \ne 0$.
  If $\cT_D$ was not slope-stable we would have $r \le q$.
  
  According to \cite[Theorem
  5.3]{du_Plessis-Wall:discriminants},  the integer $r$ satisfies
  $(d-r-1)(d-1)^{N-1} \le \deg(\sing(D))$. 
  Hence, if $\cT_D$ was not slope-stable then $r \le q$, so
  $(d-q-1)(d-1)^{N-1} \le \deg(\sing(D))$, which contradicts our assumption.
\end{proof}

\begin{Remark}\label{rmk-cone}
An obvious obstruction to stability of $\cT_D$ is that a partial
derivative of the equation $f$ defining $D$ vanishes identically, in a
suitable system of coordinates.
Indeed, if this happens then the sheaf $\cT_D$ admits a decomposition of the following type:
\[
\cT_D \simeq {\cT}_{\tilde{D}}  \oplus r.\cO_{\PP^N},
\]
where $r$ denotes the number of vanishing derivatives. This excludes
that $\cT_D$ is slope-semistable, or simple.

In characteristic zero this is equivalent to the fact that $D$ is a
cone, where $r-1$ equals the dimension of the linear (projective)
subspace which is the apex of the cone.
\end{Remark}

\begin{Remark}
If no partial
derivative of $F$ vanishes identically (up
to a coordinate change) and
 $(d-1)(s+1) < N$, then the
sheaf $\cT_D$ is slope-stable by Corollary \ref{corstabileH0}.
Notice that this numerical condition is sharp, as the following
example shows. Assume $\cha(\kk)$ does not divide $d$ nor $d-1$ and
consider the hypersurface $D$ defined by the homogeneous polynomial: 
\[
F = x_0 x_1^{d-1} + \sum_{j=2}^N x_j^d, \:\:\mbox{ which gives } \:\:
\nabla(F) = [x_1 ^{d-1}, \:\: (d-1)x_0x_1^{d-2}, \:\: d x_2^{d-1}, \:\:
\ldots, \:\: d x_N^{d-1} ]. 
\]

Observe that $D$ is singular only
at the point $(1:0:\ldots:0)$.
The associated sheaf $\cT_D$ has $\rH^0(\cT_D)=0$ and $\rH^0(\cT_D(1))\neq
0$. If $d \geq N+1$, we get $c_1(\cT_D) \leq
-\rk (\cT_D)$, so the sheaf $\cT_D$ is not slope-stable because
$\cO_{\PP^N}(-1) \subset \cT_D$ is a destabilizing subsheaf.

\end{Remark}

\subsection{A Torelli-type result}

\label{section:torelli}

In this section we will focus on some results of ``Torelli
type''. Recall that such nomenclature is used in general for results
on the embeddings between moduli spaces. In particular, this type of
problems for logarithmic tangent sheaves has been proposed
by Dolgachev and Kapranov in
\cite{dolgachev-kapranov:arrangements}, followed by many others. 

In our case case we are interested in the morphism which associates to
a hypersurface $D \subset \PP^N$ its logarithmic tangent sheaf
$\cT_D$. More specifically we are interested in the following
question: \textit{Does the logarithmic tangent sheaf $\cT_D$ determine the hypersurface $D$?}
These hypersurfaces have been called \textit{DK-Torelli} in
\cite{dimca:jacobian-syzygies}.  Keeping this definition, we provide
an extension of \cite[Theorem 1.5]{dimca:jacobian-syzygies} to the
case of non-isolated singularities, with a similar proof. For
terminology about Sebastiani-Thom hypersurfaces and multiplicity of
singularities we refer to \cite{wang:jacobian}.

Before stating the result, recall that, by \cite{sernesi:jacobian-documenta}, if $D$ is singular there is an identification of $R$-modules:
\[
  \rH^1_*(\cT_D) \simeq \frac{J_D^\sat}{J_D},
\]
where $J_D^\sat$ denotes the saturation of $J_D$. In particular we have isomorphisms:
\[
\rH^1(\cT_D(t-d)) \simeq \left( \frac{J_D^\sat}{J_D} \right)_t, \qquad
\mbox{for all $t \in \ZZ$},
\]
and these commute with multiplication maps in $R$.

\begin{Proposition} \label{prop:torelli}
Suppose that there exists an integer $m<d-1$ such that:
\begin{itemize}
\item $\rH^0(\cT_D(2m))=0$;
\item there exist two elements in $\rH^1(\cT_D(m-d))$ that admit two
  non zero representatives $h_1,h_2 \in (J_D^\sat)_m$ with no common factor.
 \end{itemize}
 Then, $\cT_D$ determines the Jacobian ideal of $F$.\\
 Furthermore, if the hypersurface $D$ is not DK-Torelli, then $D$ has a singularity of multiplicity $d-1$ or the polynomial $F$ is of Sebastiani-Thom type. 
\end{Proposition}
\begin{Remark}
Notice that having a singularity of multiplicity $d-1$ is not a
sufficient condition for the hypersurface $D$ to fail
the DK-Torelli property. Indeed, as we will see in \S
\ref{sec-DKTorelli-det}, the hypersurfaces defined by determinants
satisfy the DK-Torelli property but do admit such singularities -- we
will actually exploit these singularities in
\S \ref {subsection:semigeneric}.

On the contrary, being of Sebastiani-Thom type is a
sufficient condition for failing the DK-Torelli property, at least
when $D$ is not a cone. This can be seen taking the following
description of the defining equation for a Sebastiani-Thom divisor:
$$
F = G(x_0,\ldots,x_i) + H(x_{i+1},\ldots,x_N),
$$
with $G$ and $H$ non zero homogeneous polynomials of degree $d$, and consider the infinite family
$$
F_\alpha = \alpha G(x_0,\ldots,x_i) + H(x_{i+1},\ldots,x_N) \:\:\mbox{ with } \:\: \alpha \in \kk.
$$
Having that the Jacobian ideals $J_{D_\alpha}$ coincide, all the hypersurfaces $D_\alpha = \VV(F_\alpha)$ induce the same logarithmic sheaf and therefore are not DK-Torelli.
\end{Remark}
\begin{proof}[Proof of Proposition \ref{prop:torelli}] 
Our first goal is to characterize when a homogeneous polynomial $g$ of
degree $d-1$ belongs to the Jacobian ideal $J_F$. In order to
do so, consider, for any $k\in \NN$, the following exact sequence: 
$$
0 \to \cO_{\PP^N}(-d+1+k) \xrightarrow{\cdot g} \cO_{\PP^N}(k) \to \cO_Y(k) \to 0,
$$
with $Y = \VV(g)$. Let us tensor it by $\cT_D$ and note that the
first map remains injective, since $\cT_D$ is torsion free, and consider
the induced exact sequence in cohomology:
$$
\begin{array}{rl}
0 &\rightarrow \rH^0(\cT_D(-d+1+k)) \rightarrow \rH^0(\cT_D(k)) \rightarrow \rH^0(\cT_D(k)|_{Y}) \rightarrow \vspace{2mm}\\
&\rightarrow \rH^1(\cT_D(-d+1+k)) \rightarrow \rH^1(\cT_D(k)) \rightarrow \cdots
\end{array}
$$

We know that $\rH^0(\cT_D(k))$ describes the homogeneous syzygies of
degree $k$ of the Jacobian ideal, that is, it is given by all the
$(N+1)$-tuples $(a_0,\ldots,a_N)$ of homogeneous polynomials of degree
$k$ such that $\sum_{i=0}^N a_i \frac{\partial F}{\partial x_i} =0$.
This implies that the first linear map of the previous diagram
does not depend on the choice of $g$.
 
Consider
 the multiplication map by $g$:
$$
\left(\frac{J_D^\sat}{J_D}\right)_{m} \stackrel{(\cdot g)_m}{\longrightarrow} \left(\frac{J_D^\sat}{J_D}\right)_{m+d-1}.
$$

It is straightforward to observe that, if $g \in J_D$, then
$(\cdot g)_m=0$.
Let us prove the converse implication. Suppose thus that $(\cdot g)_m=0$ and
note that that both $g \cdot h_1$ and $g \cdot h_2$ belong to
$\left(J_D\right)_{m+d-1}$. This means that there are $(N+1)$-tuples
of homogeneous polynomials $(a_0,\ldots,a_N)$  and $(b_0,\ldots,b_N)$ of degree $m$, such that:
$$
g \cdot h_1 = \sum_{i=0}^N a_i \frac{\partial F}{\partial x_i} \:\:\mbox{ and } \:\: g \cdot h_2 = \sum_{j=0}^N b_j \frac{\partial F}{\partial x_j}.
$$

Therefore we get:
$$
0 = (g \cdot h_1)h_2 - (g \cdot h_2)h_1 = \sum_{j=0}^N(a_j h_2-b_j h_1)\frac{\partial{F}}{\partial x_j}.
$$

In view of the assumption $\rH^0(\cT_D(2m))=0$, we have thus:
$$
a_j h_2-b_j h_1 =0, \:\: \mbox{ for all } \:\: j=0,\ldots,N.
$$
Since $h_1$ and $h_2$ have no common factor, we have that $h_1 | a_j$
and $h_2 | b_j$, for $j=0,\ldots,N$. In turn, this implies that $g \in J_D$.
\medskip

Summing up, a polynomial $g$ of degree $d-1$ lies in $J_D$ if and only
if $(\cdot g)_m=0$. Since $J_D$ is generated in degree $d-1$, this
says that $J_D$ is recovered by the $R$-module
structure of $\rH^1_*(\cT_D)$, so that $J_D$ is determined by $\cT_D$.
For the last part of the statement, once we have proven that
$\cT_D$ determines the Jacobian ideal, we apply \cite[Theorem 1.1]{wang:jacobian}. 
\end{proof}

\section{Symmetric determinants}

\label{section:symmetric_determinants}
In this section we suppose that the field $\kk$ is of characteristic different from $2$.
Fixing an integer $n \ge 2$, we describe the ring $R$ as $R=\kk[x_{i,j} \mid 1 \le i \le j \le n]$, hence $N={n+1 \choose 2}-1$. For $1 \le i
\le j \le n$, put $x_{j,i}=x_{i,j}$ and let $\MM$ be the matrix
$(x_{i,j})_{1 \le i,j \le n}$.
Consider $F=\det(\MM)$.
The \textit{generic} symmetric determinant is the degree-$n$
hypersurface $D =\VV(F) \subset \PP^N$. It is singular along the
subscheme $\sing(D)$ cut by the $N+1$ minors of order $n-1$ of $\MM$
obtained by removing from $\MM$ the $i$-th line and $j$-th column, with
$1 \le i \le j \le n$.
Moreover,  $\sing(D)$ has codimension $3$ in $\PP^N$.

Consider now a projective plane $P \subset \PP^N$. The vector bundle
of $k$-th principal parts $\cE_k$ is defined as kernel of the natural
evaluation of sections $\rH^0(\cO_P(k)) \otimes \cO_P \to \cO_P(k)$.
The main goal of this section is to prove Theorem \ref{main:symmetric_determinant}, which establishes a link
between $\cT_D$ and $\cE_{n-1}$ that yields the stability of $\cT_D$.

\subsection{Proof of Theorem \ref{main:symmetric_determinant}}
  Define the graded algebra $A$ as quotient of $R$ by the
  homogeneous ideal generated by the minors of order $n-1$ of $\MM$.
  The minimal graded free resolution of $A$ is given by
  the Goto–J\'{o}zefiak–Tachibana
  complex, see \cite{jozefiak:symmetric,goto-tachibana}, see also
  \cite[\S 6.3.8]{weyman:tract}. This takes the form:
  \begin{equation}
    \label{joze}
    0 \la A \la R \la {n+1 \choose 2}.R(1-n) \la (n^2-1).R(-n) \la 
    {n \choose 2}.R(-1-n) \la 0,
  \end{equation}
  where the kernel of $A \la R$ is generated by the partial
  derivatives of $F$.
  The module $T_D(1-n)$ is the kernel of the resulting map $R \la {n+1
    \choose 2}.R(1-n)$ so its resolution is the truncation of the
  above resolution at the middle step. Upon sheafification, this gives item \ref{thereso}.
  
  Next, note that \ref{isstable} follows from
  \ref{isprincipal}. Indeed, by \cite{re:principal}, the vector bundle
  $\cE_{n-1}$ is slope-stable. Now, if $\cT_D$ had a destabilizing subsheaf
  $\cK$, then choosing $P$ to be a generic plane, transverse to the locus 
  where $\cT_D/\cK$ fails to be locally free, we would get a subsheaf
  $\cK|_P \subset \cT_D|_P$ with the same rank and slope as $\cK$, so
  that $\cK|_P$ would destabilize $\cE_{n-1}$, a contradiction.

  So it remains to prove \ref{isprincipal}. Note that, since $\sing(D)$ has
  codimension $3$ in $\PP^N$, we may choose $P$ disjoint from $\sing(D)$ so
  that $\cT_D|_P$ fits into:
  \begin{equation}
    \label{TDP}
      0 \to \cT_D|_P \to (N+1).\cO_P \to \cO_P(n-1) \to 0.    
  \end{equation}

  Note that $N+1=\rh^0(\cO_P(n-1))$ and observe that precomposing the
  evaluation of sections $\rH^0(\cO_P(n-1)) \otimes \cO_P \to
  \cO_P(n-1)$ with an automorphism of $(N+1).\kk$ we get a kernel
  bundle which is isomorphic to $\cE_{n-1}$. So it suffices to prove
  that, for generic $P$, the map $(N+1).\cO_P \to \cO_P(n-1)$
  appearing in \eqref{TDP} is the evaluation of global sections, up to
  precomposing with an isomorphism. But all such maps are the same up
  to precomposing with an isomorphism provided that they have maximal rank, so
  it is enough to prove that for generic $P$ we have
  $\rH^0(\cT_D|_P)=0$, or equivalently $\rH^1(\cT_D|_P)=0$.

  To achieve this, consider the coordinate ring $R_P$ of $P$.
  Taking the quotient by the homogeneous ideal generated by
  partial derivatives of $F$ we obtain a graded algebra of dimension $N-2$.
  Passing to the quotient modulo the ideal of $P$ we get thus an
  Artinian algebra $A$, whose resolution is just obtained by
  specialization of \eqref{joze}, hence:
  \[
    0 \la A \la R_P \la {n+1 \choose 2}.R_P(1-n) \la (n^2-1).R_P(-n) \la 
    {n \choose 2}.R_P(-1-n) \la 0.
  \]

  We get, for all $t \in \ZZ$, $\rH^1(\cT_D|_P(t-n+1)) \simeq
  A_t$. Computing dimension in the above display gives
  $A_t=0$ for all $t \ge n-1$ so $\rH^1(\cT_D|_P)=0$ and we are done.

\section{Determinants}

This section is devoted to the proof of stability of the logarithmic
tangent sheaf of the determinant divisor of a matrix of
indeterminates.
We work over an arbitrary field $\kk$. 

\subsection{Basic setup}

\label{section:determinants}
Let us fix an integer $n \ge 2$.
  Consider the graded ring $R=\kk[x_{i,j} \mid
1 \le i , j \le n]$ as coordinate ring of $\PP^N$ with $N=n^2-1$.
We call \textit{tautological determinant} the form:
\[
  F = \det((x_{i,j})_{1 \le i,j \le n}).
\]
The corresponding tautological determinantal hypersurface of
degree $n$ is the divisor:
\[
  D=\VV(F).
\]
The hypersurface $D \subset \PP^N$ is singular along the
subscheme $\sing(D)$ defined by Jacobian ideal of $F$, which in turn is
generated by the $n^2$ minors of order $n-1$ of the tautological
matrix of variables $(x_{i,j})_{1 \le i,j \le n}$.
The subscheme $\sing(D)$ has codimension $4$ in $\PP^N$.\\

In light of the following remark, all the constructions proposed in the remaining of this section can be considered for $n\ge 3$. Moreover, it explains why we have to suppose $n\ge 3$ for \S \ref{section:familiesdeterminants} as well.

\begin{Remark}\label{rmk-case2}
Consider the case $n=2$, for which $D$ is a smooth quadric in $\PP^3$. It is known (reported for example in \cite{angelini}) that the logarithmic vector bundle associated to any smooth quadric  in $\PP^3$ is $\Omega_{\PP^3}(1)$. Hence, in this case, the DK-Torelli property does not hold. 

The same happens when considering the symmetric case, in which the obtained divisors are smooth conics. Moreover, their defining equation is of Sebastiani-Thom type, that ensures once more that they are not DK-Torelli.
\end{Remark}

\subsubsection{Section outline}\label{sec-outline}
Our goal is to prove Theorem \ref{main:determinant}. Here are our main steps:
\begin{enumerate}[label=\roman*)]
\item \label{sketch-i} find a resolution of the module of global sections $T_D$;
\item \label{sketch-ii} prove that the logarithmic sheaf restricts to
  a quadric surface $X$, with $X \cap \sing(D) = \emptyset$, to the  bundle of
  principal parts $\cE_{n-1} = \ker(\rH^0(\cO_X(n-1))) \otimes \cO_X
  \to \cO_X(n-1)$;
\item \label{sketch-iii} prove that the principal part bundle $\cE_{n-1}$ of the quadric surface is slope-stable.
\end{enumerate}

\S \ref{section:logdet} is devoted to prove item \ref{sketch-i}.
In \S \ref{subsection:semigeneric}, we introduce the concept of
\textit{semigeneric matrix} which allows us, through a quadratic
Lefschetz property described in \S \ref{section:Lefschetz}, to
prove that $(\cT_D)_{|X} \simeq \cE_{n-1}$.
\S \ref{section:stabilityprincipalbundle} is devoted to prove
that $\cE_{n-1}$ is slope-stable. Finally, in \S
\ref{section:provethm5}, we combine all of the previous results to
prove that $\cT_D$ is slope-stable as well.

\subsubsection{An intrinsic setup} \label{intrinsic}

Let $U,V$ be two $n$-dimensional $\kk$-vector spaces and set:
\[
  \AAA = \Hom_\kk(V,U) \simeq V^* \otimes U.
\]
We identify $\PP^N$ with $\PP(\AAA)$, so an element
$[\ra]$ of $\PP(\AAA)$ is the proportionality class of $\ra \in \AAA^*
\simeq V \otimes U^*$,
that is, of a non-zero linear map $\ra : U \to V$.

An element of $[\rf]$ of $\PP(\End_\kk(\AAA))$ is thus the proportionality
class of an element $\rf \in \End_\kk(\AAA)$, which under the
identification $\rH^0(\cO_{\PP(\AAA)}(1))=\AAA$ can be seen as a map:
  \[
    M_\rf : U\otimes \cO_{\PP(\AAA)}(-1) \to V\otimes \cO_{\PP(\AAA)}.
  \]

  We denote by $\ri$ the special element $\ri=\id_\AAA \in
  \End_\kk(\AAA)$. In any given basis $(u_i\mid 1 \le i \le n)$ and
  $(v_i\mid 1 \le i \le n)$ of $U$ and $V^*$,
  setting $(x_{i,j} \mid 1 \le i,j \le n)$ for the dual basis of
  $(u_i \otimes v_j \mid i,j \le n)$, the matrix of $M_\ri$ is  
  $(x_{i,j})_{1 \le i,j \le n}$, so the tautological determinant is $D=D_\ri$.
 
\subsection{The resolution of $T_D$} \label{section:logdet}

Here we give a minimal graded free resolution of the graded module $T_D$ associated with
the sheaf $\cT_D$.
The resolution is obtained directly as a truncation of the
Gulliksen-Neg\r ard's complex. The upshot is that the projective
dimension of $T_D$ is $2$, one less than the codimension in $D$ of the
singular locus of $D$, in analogously with free divisors.

The divisor $D=\VV(\det(M_\ri)) \subset \PP(\AAA)$ is invariant with
respect to the action of the group $G=\SL(U) \times \SL(V)$ on
$\PP(\AAA)$. We seek a resolution of $T_D$ which is equivariant for
the induced action of $G$ on the polynomial ring $R$ seen as the
symmetric algebra of $\AAA$.

\begin{Proposition}\label{res-determinant}
  There is a minimal graded free  $G$-equivariant resolution of $T_D$
  of the form:
  \begin{equation}
    \label{varphi}
    0 \la T_D \la (\ssl(U) \oplus \ssl(V))\otimes R(-n)
    \xleftarrow{\varphi}  \AAA \otimes R(-n-1) \la R(-2n) \la 0.
  \end{equation}
    
\end{Proposition}

Looking only at the homogeneous Betti numbers, the resolution reads:
  \[
    0 \la T_D \la 2(n^2-1).R(-n)
    \la n^2.R(-n-1) \la R(-2n) \la 0.
  \]

\begin{proof}
  Recall that the homogeneous Jacobian ideal $J_D$ is defined by the partial
  derivatives of $F=\det(M_\ri)$, where the matrix of the map $M_\ri$ is the 
  matrix of indeterminates $(x_{i,j})_{1 \le i,j \le n}$.
  This ideal is generated by the $n^2$ minors of order
  $n-1$ of $M_\ri$. Namely, there is a natural surjective map :
  \begin{equation}
    \label{IZD}
     \AAA^* \otimes R(1-n) \to J_D.
  \end{equation}

  This map is the first differential of the resolution of $J_D$ given
  by the Gulliksen-Neg\r ard complex, see \cite{gulliksen-negard} or
  \cite[\S 6.1.8]{weyman:tract}. This is a $G$-equivariant resolution
  that reads: 
  \[
  \begin{array}{rcl}
&  \ssl(U) \otimes R(-n)  \\ 
  0 \la J_D \la  \AAA^* \otimes R(1-n) \la    & \oplus & \la  \AAA \otimes  R(-n-1)
                                   \la R(-2n) \la 0. \\
   &  \ssl(V)\otimes R(-n) 
  \end{array}
  \]
  
  The resolution of $T_D$ is obtained 
  by truncation of the resolution of
  $J_D$, since $T_D$ is the kernel of the map \eqref{IZD}.
\end{proof}

\subsection{Simplicity of $\cT_D$} \label{section:simplicity}

  We consider the determinant hypersurface $D$.
  We would like to prove the following result.
  \begin{Lemma} \label{simple lemma}
    The sheaf $\cT_D$ is simple.
  \end{Lemma}

  \begin{proof}
Consider the singular locus $Z=\sing(D)$ and the ideal sheaf $\cI_Z=\cI_{Z/\PP(\AAA)}$.
Sheafifying the surjection \eqref{IZD} we get the exact sequence:
\begin{equation}
  \label{syz}
  0 \to \cT_D \to \AAA \otimes \cO_{\PP(\AAA)} \to \cI_Z(n-1) \to 0,
\end{equation}
where the surjection onto $\cI_Z(n-1)$ is the natural evaluation of
global sections, so that:
\begin{equation}
  \label{due}
\rH^0(\cT_D)=\rH^1(\cT_D)=0.  
\end{equation}
Since the surjection in \eqref{syz} is the sheafification of the epimorphism of
graded $R$-modules $\AAA \otimes R \to I_Z(n-1)$, actually we have
$\rH^1_*(\cT_D)=0$.
Next, recall that $R_Z=R/I_Z$ is a graded Cohen-Macaulay ring
of dimension $N-4$. This implies:
\begin{equation}
  \label{vanishingI}
  \rH^p_*(\cI_Z)=0,\qquad \mbox{for $p\in \ZZ \setminus \{0,N-3,N\}$}.
\end{equation}
Together with $\rH^1_*(\cT_D)=0$, 
this gives:
\begin{equation}
  \label{vanishingII}
  \rH^p_*(\cT_D)=0,\qquad \mbox{for $p\in \ZZ \setminus \{0,N-2,N\}$}.
\end{equation}

Next, note that Serre duality gives a natural isomorphism:
\begin{equation}
  \label{serreI}
  \rH^{N-p}(\cI_Z(t-N-1))^* \simeq
  \Ext^p(\cI_Z(t),\cO_{\PP(\AAA)}),\qquad \mbox{for all $p,t \in \ZZ$}.
\end{equation}

This implies $\Ext^1_{\PP(\AAA)}(\cI_Z(n-1),\cO_{\PP(\AAA)})=0$.
Also, we have
$\Hom_{\PP(\AAA)}(\cI_Z(n-1),\cO_{\PP(\AAA)})=\rH^0(\cO_{\PP(\AAA)}(1-n))=0$. 
Hence, 
applying
$\Hom_{\PP(\AAA)}(\cI_Z(n-1),-)$ to \eqref{syz},  we get:
\[
  \kk \simeq \End_{\PP(\AAA)}(\cI_Z) \simeq \Ext^1_{\PP(\AAA)}(\cI_Z(n-1),\cT_D).
\]
Applying $\Hom_{\PP(\AAA)}(-,\cT_D)$ to \eqref{syz} and using
\eqref{due}, we get:
\[\End_{\PP(\AAA)}(\cT_D) \simeq
  \Ext^1_{\PP(\AAA)}(\cI_Z(n-1),\cT_D) \simeq \kk.
\]
The lemma is proved.
  \end{proof}

\subsection{Semigeneric matrices} \label{subsection:semigeneric}

The next step is to choose a linear section $L \simeq \PP^3$ of $\PP^N$ which is
semigeneric, in a sense that we will make more precise in the next
paragraph. The goal of this partial genericity will be to ensure that,
for a honestly generic choice of $L$, the resulting quotient algebra
$A$ is Artinian and satisfies the quadratic
Lefschetz property, as we will see in \S \ref{section:Lefschetz}.

Restricting $M_\ri$ to $L$, 
we get an $n \times n$ matrix $M_L$ of linear forms on $L$, whose
$n^2$ minors of order $n-1$ generate the ideal $I_L \subset R_L = \kk[x_0,x_1,x_2,x_3]$ defining $A$.
Set $\fm_0=(x_1,x_2,x_3)$ and $\fm=(x_0,x_1,x_2,x_3)$.
In the next definition, we choose a basis of $U$ and $V$.

\begin{Definition}
We say that $L$ is a semigeneric section if:
\[
  M_L : n.\cO_L(-1) \to n.\cO_L \qquad \mbox {satisfies} \qquad M_L = M_0+x_0E_{1,1},
\]
where $M_0$ is generic in $R_0=\kk[x_1,x_2,x_3]$ and $E_{1,1}$ is the
elementary matrix $(E_{1,1})_{i,j}=\delta_{i,1}\delta_{j,1}$,
namely the forms ${(M_L)}_{i,j}$ lie
outside a Zariski closed subset of the set of all $n^2$-tuples of
$1$-forms in $R_0$, except for ${(M_L)}_{1,1}$ which also involves $x_0$.
In such case, we say that $M_L$ is a linear semigeneric matrix of size $n$.  
\end{Definition}

The goal of this subsection is to prove the following result. 

\begin{Proposition} \label{propn2}
  Let $M_L$ be a linear semigeneric matrix of size $n$. Then:
  \[
    I_L = x_0\fm_0^{n-2} + \fm_0^{n-1}.
  \]
\end{Proposition}

\subsubsection{Semigeneric matrices and the blown-up plane}

Our first observation aimed at proving Proposition \ref{propn2} is that the determinant of a semigeneric matrix defines a
model of the blown-up plane at $n(n-1)$ points.

Put  $p_0=(1:0:0:0)$.
Define the threefold $T = \PP(\cO_{\PP^2}(1) \oplus
\cO_{\PP^2})$ and the natural maps $\pi : T \to \PP^2$ and
$\sigma : T \to \PP^3$ so that $\pi$ is the tautological
$\PP^1$-bundle and $\sigma$ is blow-up of $\PP^3$ at $p_0$.
Set $\fl$ (resp. $\fh$) for the pull-back to $T$ of a
hyperplane in $\PP^2$ (resp. in $\PP^3$).

\begin{Lemma}\label{lemma-buplane}
  If $M_L$ is semigeneric, the degree-$n$ surface $S=\VV(\det(M_L)) \subset L \simeq \PP^3$
  is the image of $\PP^2$ by the linear system of
  curves of degree 
  $n$ passing through a smooth complete intersection of $n(n-1)$
  points. The surface $S$ is smooth away from the $(n-1)$-tuple 
 $p_0$ and has a natural desingularization $\hat S$ which is an
 element of the linear system $|\cO_T((n-1)\fl+\fh))|$.
\end{Lemma}

\begin{proof}
  The shape of $M$ implies, by multilinearity of the determinant:
  \[
    \det(M)=\det(M_0)+x_0\det(M_1), \qquad \mbox{with:} \qquad M_1=(M_0)_{2 \le i,j \le n}.
  \]
  
  Therefore $p_0$ is a point of multiplicity $n-1$ of $S$.
  Working over $\PP^2=\Proj(R_0)$ we note that,
  if $M_0$ is general enough, we may assume that the curves $G_0=\VV(\det(M_0))$ and
  $G_1=\VV(\det(M_1))$ in $\PP^2$ are smooth of degree $n$ and $n-1$ and meet
  transversely 
  at a subscheme $W \subset \PP^2$ consisting of $n(n-1)$ reduced
  points. We have:
  \begin{equation}
    \label{resolutionW}
    0 \to \cO_{\PP^2}(1-n) \to \cO_{\PP^2}(1) \oplus     \cO_{\PP^2} \to I_{W/\PP^2}(n) \to 0.
  \end{equation}

  Therefore $S$ is smooth away from $p_0$ and the projection away from $p_0$ sends $S$ birationally
  to $\PP^2$. The inverse map is defined by the complete linear system
  $|\cI_{W/\PP^2}(n)|$.

  Define the surface $\hat S$ as the blow-up of $\PP^2$ at $W$.
  We have $\hat S \simeq \PP(\cI_{W/\PP^2}(n))$, where the linear
  system associated to the tautological relatively ample
  line bundle sends $\hat S$ to $S \subset \PP^3$.
  The smooth surface $\hat S$ sits canonically in $T$, the embedding
  being defined by the projectivization of the surjection in \eqref{resolutionW}.
\end{proof}

The restriction of $\sigma$, $\pi$, $\fl$ and $\fh$ to $\hat S$ define
objects which we denote by the same letters. We set $\fe_\pi$ for the exceptional
divisor of $\pi : \hat S \to \PP^2$, hence we have:
\begin{equation}
  \label{formule-h}
  \fh = n\fl-\fe_\pi.
\end{equation}
For all $i,j \in \ZZ$, the cohomology of $\cO_T(i \fh + j \fl)$ is:
\begin{equation}
  \label{linebundlesT}
  \rH^k(\cO_T(i \fh + j \fl)) \simeq \left\{
    \begin{array}{ll}
      \bigoplus_{0 \le u \le i} \rH^k(\cO_{\PP^2}(j+u)), & \mbox{if $i \ge 0$},\\
      \bigoplus_{i+1 \le u \le -1} \rH^{k+1}(\cO_{\PP^2}(j+u)), & \mbox{if $i \le -2$},\\
      0, & \mbox{if $i = -1$}.
    \end{array}\right.
\end{equation}
Also, the cohomology of $\cO_{\hat S}(i \fh + j \fl)$ is computed via \eqref{linebundlesT}
by
induction on $i$ and $j$ from
the exact sequence:
\begin{equation}
  \label{linebundlesS}
  0 \to \cO_T((1-n)\fl -\fh) \oplus \cO_T(\fl-\fh) \to \cO_T(\fl) \oplus   \cO_T \to \cO_{\hat S}(\fh) \to 0.  
\end{equation}

\subsubsection{Linear determinantal representation of the blown-up plane}

Set $\cL=\coker(M_L)$ and write:
\begin{equation}
  \label{cokerML}
  0 \to n.\cO_L(-1) \xrightarrow{M_L} n.\cO_L \to \cL \to 0.
\end{equation}

Here, $\cL$ is a reflexive sheaf of rank $1$ on $S$ which is not
locally free.
The next lemma allows to determine an exact sequence on $\hat S$, where the rightmost term is a line bundle $\hat \cL$ on $\hat S$. Moreover, it is possible to lift such sequence to the threefold $T$, in order to define $\hat \cL$ as
the determinant of a map of bundles or rank-$n$ over $T$, whose
push-forward to $\PP^3$ is the semigeneric matrix $M_L$ we started with.

\begin{Lemma} \label{hatcM}
There is a subset $\fe_1,\ldots,\fe_m$, with $m=n(n-1)/2$, of the components of 
$\fe_1 + \ldots + \fe_{n(n-1)} = \fe_\pi$, such
that $\hat \cL = (n-1)\fl - \fe_1 - \cdots -\fe_m$ fits into: 
\begin{equation}
  \label{cokerMLn}
  0 \to \cO_T(-\fh) \oplus (n-1). \cO_T(-\fl) \to n. \cO_T \to \hat \cL \to 0.   \end{equation}
Moreover, the push-forward to $L \simeq \PP^3$ of the above sequence is \eqref{cokerML}.
\end{Lemma}

\begin{proof}
Let us use the notation described in \S \ref{intrinsic} and also in the proof of Lemma \ref{lemma-buplane}. We therefore consider the
matrices $M_L$ and $M_0$ as morphisms:
\[M_L : U \otimes \cO_{L}(-1) \to V \otimes \cO_{L},
  \qquad 
  M_0 : U \otimes \cO_{\PP^2}(-1) \to V \otimes \cO_{\PP^2}.
\]
The sheaf $\cL_0 = \coker(M_0)$ is 
a line bundle supported on the curve $G_0 \subset \PP^2$. Transposing
$M_0$, by Grothendieck duality we get:
\[
    0 \to V^* \otimes \cO_{\PP^2}(-1) \to U^* \otimes \cO_{\PP^2} \to
    \cL_0^*(n-1) \to 0.
\]
We observe that a non-zero global section $u$ of $\cL_0^*(n-1)$ is given
uniquely by a 
non-zero element $u \in U^*$ and provides thus a $1$-codimensional
quotient $U^*_u=U^* / \kk u$. The section $u$ vanishes along   
a subscheme $W_u$ of $\PP^2$ of length
$m=n(n-1)/2$ which is contained in $G_0$ and we have a resolution:
\begin{equation}
  \label{resolution_IW}
  0 \to U_u \otimes \cO_{\PP^2}(-1) \to V \otimes \cO_{\PP^2} \to
  \cI_{W_u/\PP^2}(n-1) \to 0.
\end{equation}
Since $M_L$ is semigeneric, there are two $1$-dimensional
 marked subspaces of $V$ and $U^*$, corresponding to the first row and column of $M$.
We choose $u \in U^*$ to lie in this space, so that $W_u$ is a
subscheme of half the length of $G_0 \cap G_1 = W$.

\medskip

Set $\fe_u$ for the union of the $m$ components
$(\fe_1,\ldots,\fe_m)$ of $\fe_1 + \ldots + \fe_{n(n-1)} = \fe_\pi$ which are contracted to $W_u$ by
$\pi$ and write $\fe_{\bar u}=\fe_\pi - \fe_u$.
Put $\hat \cL=\cO_{\hat S}((n-1)\fl-\fe_u)$. Pulling
back \eqref{resolution_IW} to $\hat S$ via $\pi$ and removing the
torsion part $\cO_{\fe_u}(-1)$ of $\pi^*(\cI_{W_u/\PP^2}(n-1))$  we get the exact
sequences:
\begin{align}
  \label{hatM1}&0 \to \cK \to V \otimes \cO_{\hat S} \to \hat \cL
    \to 0,\\
  \label{hatM2} &0 \to U_u\otimes \cO_{\hat S}(-\fl) \to \cK \to \cO_{\fe_u}(-1) \to 0.
\end{align}
where the surjection in the first sequence is the natural evaluation
of global sections of $\hat \cL$
and $\cK$ is defined as the kernel of this map.
Since $\cO_{\fe_u}(\fh) \simeq \cO_{\fe_u}(1)$, we have:
\[
0 \to \cO_{\hat S}(-\fh-\fe_u) \to  \cO_{\hat S}(-\fh) \to
\cO_{\fe_u}(-1) \to 0.
\]
Observe that $\cO_{\hat S}(-\fh-\fe_u) \simeq \hat \cL((1-n)\fl-\fh)$. By
\eqref{linebundlesT} we have $\rH^1(\cO_{\hat S}(\fh-\fl))=0$, hence the above surjection lifts to
$\cO_{\hat S}(-\fh) \to \cK$. Patching this with \eqref{hatM2} we
get:
\begin{equation}
  \label{cK}
  0 \to  \hat \cL((1-n)\fl-\fh) \to \cO_{\hat S}(-\fh) \oplus U_u \otimes
  \cO_{\hat S}(-\fl) \to \cK \to 0.
\end{equation}
  We rewrite this as a long exact sequence: 
  \begin{equation}
    \label{lunga1}
  0 \to  \hat \cL((1-n)\fl-\fh) \to \cO_{\hat S}(-\fh) \oplus U_u \otimes
  \cO_{\hat S}(-\fl) \to V \otimes \cO_{\hat S} \to \hat \cL \to 0,
  \end{equation}
where the sheaf $\cK$ is the image of the middle map.
Such map lifts to the threefold $T$ and we have the determinantal
representation of $\hat \cL$ :
\begin{equation}
  \label{detM}
  0 \to \cO_T(-\fh) \oplus U_u \otimes \cO_T(-\fl) \to V \otimes \cO_T \to \hat \cL \to 0.  
\end{equation}
This is precisely \eqref{cokerMLn}. Also, we have
$\sigma_*(\cO_T(-\fl)) \simeq \sigma_*(\cO_T(-\fh)) \simeq \cO_L(-1)$ and
$\sigma_*(\cO_T) \simeq \cO_L$. The functor $\sigma_*$ sends maps
$\cO_T(-\fh) \to \cO_T$ to linear forms and maps
$\cO_T(-\fl) \to \cO_T$ to linear forms which vanish at $p_0$ and
each coefficient of the matrix appearing in \eqref{cokerMLn} is mapped
via $\sigma_*$ to the corresponding coefficient of $M$. Therefore,
$\sigma_*$ sends $\hat \cL$ to $\cL$ and the lemma is proved.
\end{proof}

\subsubsection{The rigid curve}

Set $\fg$ for the strict transform of $G_1$ in $\hat S$, so that $\fg$ is smooth and:
\[
  \fg \in |\cO_{\hat S}((n-1)\fl - \fe_\pi) |, \qquad \fg^2=1-n, \qquad 
  \fg \cdot \fh = 0, \qquad \rH^0(\cO_{\hat S}(\fg))
  \simeq \kk.
\]
More precisely note that $\fg+\fl \equiv \fh$, moreover
$\deg(\fh|_\fg)=0$ and $\rh^0(\hat S,\fh)=4 > 3=\rh^0(\hat S,\fl)$, so
that $\rh^0(\fg,\fh|_\fg) \ne 0$. Therefore we have:
\begin{equation}
  \label{g=-l}
  \fg|_\fg \equiv -\fl|_\fg.  
\end{equation}

We call $\fg$ the \textit{rigid curve} of $\hat S$. We 
analyze the restriction of $\cL$ and $\cK$ to the rigid curve.
We would like to prove :
\[
  \rH^0(\fg,\cK^* \otimes \hat \cL(\fg-\fl)|_\fg)=0.
\]
Write $\cN = \hat \cL|_\fg$, so 
$\cN \simeq \cO_\fg((n-1)\fl-e_u)$.
Since $\fg|_\fg \equiv -\fl|_\fg$, it suffices to prove the following lemma.
\begin{Lemma} \label{KN}
  We have:
    \begin{equation}
    \label{-2l}
    \rH^0(\fg,\cK^*|_\fg \otimes \cN(-2\fl))=0.    
  \end{equation}
\end{Lemma}

\begin{proof}
  The divisor $e_u|_\fg$ has degree $n(n-1)-m=m$ and consists of $m$
  generic points of $\fg$, so 
  $\rh^0(\fg,\cN)=n-1$ and $\rh^1(\fg, \cN)=0$. 
  Note that the defining equation of the curve $G_0$ corresponds to
  the first element $v$ 
  of the chosen basis of the space of curves $V$ of degree
  $n-1$ through $W_u$. Hence, setting $V_v=V/\kk v$, we get an
  identification $\rH^0(\fg,\cN)=V_v$ and restricting \eqref{hatM1} we get:
  \[
    0 \to \cK_0 \to  V_v \otimes \cO_\fg \to \cN \to 0, \qquad
    \cK|_\fg \simeq \cK_0 \oplus \cO_\fg.
  \]
  Here, the sheaf $\cK_0$ is defined by the sequence and the copy of $\cO_\fg$ sits in $V \otimes \cO_\fg$ as the line
  spanned by $v$. Since $\rH^0(\fg,\cN(-2\fl))=0$, we have to show:
  \[
    \rH^0(\fg,\cK_0^* \otimes \cN(-2\fl))=0.    
  \]
  Restricting \eqref{cK} to $\fg$ and using $\cO_\fg(-\fh) \simeq \cO_\fg$ we get:
  \begin{equation}
    \label{K0}
    0 \to \cN((1-n)\fl) \to U_u \otimes \cO_\fg(-\fl) \to \cK_0 \to 0.    
  \end{equation}
  We may summarize this in the following long exact sequence:
  \[
    0 \to \cN((1-n)\fl) \to U_u \otimes \cO_\fg(-\fl) \to V_v \otimes \cO_\fg \to \cN \to 0.
  \]
  This gives a presentation:
  \[
    0 \to U_u \otimes \cO_{\hat S}(-\fl)\to V_v \otimes \cO_{\hat S} \to
     \cN \to 0.
   \]
  In particular $\rH^0(\fg,\cN(-\fl))=0$. We also note that $\cN
  \simeq \pi^*(\coker(M_1))$ and that the above sequence is the pull-back to
  $\hat S$ via $\pi$ of:
  \[
    0 \to (n-1). \cO_{\PP^2}(-1)\xrightarrow{M_1} (n-1).\cO_{\PP^2} \to
    \coker(M_1) \to 0.
   \]
  
  Applying $\cHom_\fg(-,\cN_\fg(-2\fl))$ to \eqref{K0} we get:
  \[
    0 \to \cK_0^* \otimes \cN(-2\fl) \to U_u^* \otimes \cN(-\fl) \to \cO_{\fg}((n-3)\fl)
    \to 0.
  \]
  Since $\rH^0(\fg,\cN(-\fl))=0$, we get $\rH^0(\fg,\cK_0^* \otimes
  \cN(-2\fl))=0$ and we are done.
\end{proof}

\subsubsection{Matrix factorization and the proof of Proposition \ref{propn2}}

Restricting $M_L$ to $S$, by
matrix factorization we
get:
\begin{equation}
  \label{longM}
0 \to \cL(-n) \to U \otimes \cO_S(-1) \to V \otimes \cO_S  \to \cL \to 0.
\end{equation}
Recall that $\deg(S)=n$ and use \cite[Lemma 3.2]{kleppe-miro-roig:dimension} to get:
\[
\Hom_S(\cL(-n),\cL(-1)) \simeq
  \rH^0(S,\cO_S(n-1)) \simeq
  \rH^0(\PP^3,\cO_{\PP^3}(n-1)).
\]
Therefore, applying $\Hom_S(-,\cL(-1))$
to the inclusion $\cL(-n) \to U \otimes \cO_S(-1)$
  appearing in \eqref{longM} and recalling that we set $V=\rH^0(S,\cL)$, we get a linear map:
\begin{equation}
  \label{n2}
  \Hom(U,V) = \AAA^* \to \rH^0(\PP^3,\cO_{\PP^3}(n-1)).  
\end{equation}

\begin{proof}[Proof of Proposition \ref{propn2}]
  First observe that, since the degree of $x_0$ in any of the minors
  defining $I_L$ is at most $1$, we have $I_L \subset x_0\fm_0^{n-2} + \fm_0^{n-1}$.
  Also, both $I_L$ and $x_0\fm_0^{n-2} + \fm_0^{n-1}$ are generated by $n^2$
  forms of degree $n-1$, those of $x_0\fm_0^{n-2} + \fm_0^{n-1}$ being
  linearly independent. Hence it suffices to prove that the $n^2$
  generators of $I_L$ are also linearly independent.

  The linear span of the $n^2$ minors under consideration is the image
   of the map \eqref{n2}, so we have to prove that this map is injective.

Set $\cL'=\hat \cL((n-2)\fh+(1-n)\fl)$. Using \eqref{linebundlesT}, we deduce from \eqref{detM}:
\[
  \rH^0(\cL'(\fh)) \simeq \rH^0(\cL'(\fl)) \simeq V.
\]
Therefore, we get an identification:
\[
  \Hom_{\hat S}(\cO_{\hat S}(-\fh) \oplus U_u \otimes \cO_{\hat
    S}(-\fl),\cL') \simeq U^*\otimes V = \AAA^*.
\]
Also, we have:
\[
  \Hom_{\hat S}(\hat \cL((1-n)\fl-\fh),\cL') \simeq \rH^0(\cO_{\hat S}((n-1)\fh)) \simeq \rH^0(\cO_{\PP^3}(n-1)).
\]
Now, applying $\Hom_{\hat S}(-,\cL')$ to
\eqref{cK} we get an exact sequence:
\[
0 \to \Hom_{\hat S}(\cK,\cL') \to  \AAA^* \to \rH^0(\PP^3,\cO_{\PP^3}(n-1)) \to
\Ext^1_{\hat S}(\cK,\cL').
\]
In view of Lemma \ref{hatcM}, the sequence \eqref{longM} is the image
via $\sigma_*$ of \eqref{lunga1}, in particular the middle map of the above
sequence is identified via $\sigma_*$ with the map \eqref{n2}.
Thus we are reduced to prove $\Hom_{\hat
  S}(\cK,\cL')=0$, that is:
\begin{equation}
  \label{KM'}
  \rH^0(\hat S, \cK^* \otimes \cL')=0.  
\end{equation}

To do this, we use the rigid curve $\fg \equiv \fh-\fl$. Note that:
\[
  \cL' \simeq \hat\cL((n-2)\fg-\fl).
\]
By
\eqref{g=-l}, for all integer $j$ we have:
\[
  0 \to \cO_{\hat S}((j-1)\fg - \fl) \to \cO_{\hat S}(j\fg - \fl) \to
  \cO_\fg(-(j+1)\fl) \to 0.
\]
Since $\fl|_\fg$ is effective, using induction on $j$ with $1 \le j
\le n-2$, to show \eqref{KM'}  it suffices to
prove:
\[
  \rH^0(\hat S, \cK^* \otimes \hat \cL(-\fl))=0, \qquad
  \rH^0(\fg, \cK^* \otimes \hat \cL(-2\fl)|_\fg)=0.  
\]
The second vanishing is precisely Lemma \ref{KN} so we only need to
show the first one. But this follows by looking at the dual of
\eqref{hatM1}, tensored with $\hat \cL(-\fl)$, which reads:
\[
  0 \to \cO_{\hat S}(-\fl) \to V^* \otimes \hat \cL(-\fl) \to \cK^*
  \otimes \hat \cL(-\fl) \to 0,
\]
so since $\rH^0(\hat S,\cL(-\fl))=\rH^1(\hat S,\cO_{\hat S}(-\fl))=0$, we get
$\rH^0(\hat S, \cK^* \otimes \hat \cL(-\fl))=0$. This completes the proof of
Proposition \ref{propn2}.
\end{proof}

\subsection{Quadratic Lefschetz property}\label{section:Lefschetz}

Consider a linear subspace $L \simeq \PP^3 \subset \PP^N$.
We get a projection $R\to R_L$ onto a polynomial ring $R_L$ in
$4$ variables, denoting as before $R_L\simeq \kk[x_0,\ldots,x_3]$.

Choose an integer $n \ge 3$. Since the singular
locus $\sing(D)$ has codimension $4$, generically
 $L$ will not meet
$\sing(D)$. In this case the image of $J_D$ in $R_L$ defines an Artinian Gorenstein algebra $A_L$ as quotient of
$R/(J_D+I_L)$. The algebra $A=A_L$, called the Artinian reduction of $R/J_D$,
inherits a minimal graded free resolution:
\begin{equation}
  \label{eq:artinian}
    0 \la A \la R_L \la n^2.R_L(1-n) \la 2(n^2-1).R_L(-n)
    \la n^2.R_L(-n-1) \la R_L(-2n) \la 0.
\end{equation}
  We say that the algebra $A$ has the degree-$k$ Lefschetz property
  if, for each graded piece $A_t$ of $A$, there is an element $h \in A_1$ such that $\cdot h^k : A_t \to A_{t+k}$ has maximal rank.
  
If $L \cap Z \ne \emptyset$, the algebra $A=A_L$ is no longer Artinian. In
the next paragraph we will see how, choosing $L$ in a
semigeneric way, the resulting non-Artinian algebra allows to
establish the quadratic Lefschetz 
property for the Artinian algebras $A'=A_{L'}$ given by the generic choice $L' \simeq \PP^3$.

\begin{Lemma} \label{Lemma i and ii}
  Assume that there is a linear subspace $L = \PP^3\subset \PP^N$ and a linear form $h$
  on $L$
  such that $\cdot h^2 : A_{n-3} \to A_{n-1}$ is an isomorphism. Then,
  for generic choice of $L' \simeq \PP^3\subset \PP^N$, the Artinian
  algebra $A'=A_{L'}$ has the quadratic Lefschetz property.
\end{Lemma}

\begin{proof}
  Note that, since for a generic choice of
  $L'$ the Artinian algebra
  $A'$ has a graded resolution of the form \eqref{eq:artinian}, the graded
  algebra structure of $A'$ and $R_L$ coincide up to degree $n-1 \ge 1$.
  Therefore, $\cdot h^2 : A'_t \to
  A'_{t+2}$ has maximal 
  rank for any choice of $0 \ne h \in A'_1$ and all $t \in \{0,\ldots,n-4\}$.
  By Gorenstein duality, the same happens for $t \in
  \{n-2,\ldots,2n-6\}$, indeed $A'$ has socle degree $2n-4$.
  Therefore, $A'$ has the quadratic Lefschetz property 
  if there is $h \in A'_1$ such that $\cdot h^2 : A_{n-3} \to A_{n-1}$ has
  maximal rank. Because $\dim(A_{n-1})={n+2\choose 3}-n^2={n\choose
    3}=\dim(A_{n-3})$, this amounts to ask that $\cdot h^2 : A_{n-3}
  \to A_{n-1}$ is an isomorphism. Since, by our assumption, this holds for a special
  choice of the linear space
  $L$ and the element $h \in A_1$, it also holds for 
  a generic choice of the linear space $L'$ and the element $h \in A'_1$. Therefore $A'$ has the
  quadratic Lefschetz property, as required.
  
 \end{proof}
\subsection{Vector bundle of principal parts on a quadric surface}\label{section:stabilityprincipalbundle}

 In this section, we assume $\cha(\kk)=0$.
Consider the Segre product $X\simeq\PP^1 \times \PP^1 \subset \PP^3$ and,
for any $(a_1,a_2) \in \ZZ^2$, put
$\cO_X(a_1,a_1)=p_1^*\cO_{\PP^1}(a_1) \otimes p_2^*\cO_{\PP^1}(a_2)$,
where $p_1$ and $p_2$ are the two projections of $X$ onto its two
$\PP^1$ factors. Write $U_1=\rH^0(X,\cO_X(1,0))$ and
$U_2=\rH^0(X,\cO_X(0,1))$ and consider, for $n \in \NN$, the sheaf of principal parts
$\cE_n$ defined as kernel of the natural evaluation:
\[
  \cE_n = \ker\left(S^n U_1 \otimes S^n U_2 \otimes \cO_X \to \cO_X(n,n)\right).
\]

The goal of this subsection is to prove the following result.

\begin{Proposition} \label{En is stable}
  For any integer $n \ge 3$, the sheaf $\cE=\cE_n$ is slope-stable.
\end{Proposition}

Set $G=\SL_2(\kk) \times \SL_2(\kk)$ and let $P$ be the subgroup of $G$
consisting of pairs of upper triangular matrices. Then $X \simeq
G/P$.
 We will use a special case of the equivalence of categories of linearized $G$-equivariant
vector bundles on $X=G/P$ and of finite-dimensional representations of $P$.
In turn, following \cite{bondal-kapranov:homogeneous, hille:homogeneous,
  ottaviani-rubei:quiver}, these categories are equivalent to the
category of
representations of the quiver with relations $\cQ_X$ that we describe below.
For the case under consideration of $X=\PP^1 \times \PP^1$, this takes place
in spite of the fact that the group $G$ is not strictu sensu of type
$A,D,E$.

Indeed,
consider the quiver $\cQ_X$, whose
vertices are defined by the irreducible representations of the
semisimple part of $P$, isomorphic to $\kk^* \times
\kk^*$. The weight function gives an identification of the vertice $\cQ_X$ and 
points of the lattice $\ZZ^2$, independently of $\cha(\kk)$.
In terms of sheaves over $X$, a vertex $\lambda=(a,b) \in
\ZZ^2$ of $\cQ_X$ is given by $\cO_X(\lambda)$.
The arrows of $\cQ_X$ are determined by the
invariant part of the extensions between representations.
Namely, there is an arrow from 
$\lambda =(a,b) \in \ZZ^2$ to $\mu = (c,d) \in \ZZ^2$ if
$\Ext^1_X(\cO_X(\lambda),\cO_X(\mu))^G \ne 0$, which happens if and
only if $a-c=d-b+2$.
 Note that in this case we have
$\Ext^1_X(\cO_X(\lambda),\cO_X(\mu))^G = \kk$.
The infinite quiver $\cQ_X$ has four connected components,
characterized by the fact of containing $\cO_X$, or $\cO_X(1,0)$, or $\cO_X(0,1)$, or $\cO_X(1,1)$.
There relations of the quiver $\cQ_X$ are given by imposing
commutativity of all the square diagrams of the following form:
\[
  \xymatrix@R-5ex{
    \small{(a,b)}  &\small{(a+2,b)} \\
    \bullet \ar[r] & \bullet  \\
    &&\\
    &&\\
    &&\\
    \bullet \ar[r] \ar[uuuu] & \bullet \ar[uuuu] \\
    \small{(a,b+2)} & \small{(a+2,b+2)}
  }
\]

Given a $G$-homogeneous bundle $E$, there exists a
 $G$-equivariant filtration:
$$
0 \subset E_1 \subset \cdots \subset E_k = E
$$
such that $E_i / E_{i-1}$ is a line bundle. The 
 \textit{associated
  graded bundle} is defined as:
\[
\gr(E)= \bigoplus_{i} E_i / E_{i-1},
\] 
and does not depend on the chosen filtration.
Write the graded bundle
as
\[
\gr(E) = \bigoplus_{\lambda \in \ZZ^2}V_{\lambda} . \cO_X(\lambda),
\]
where $V_\lambda$ is a $\kk$-vector space whose rank is the
number of copies of $\cO_X(\lambda)$ in $\gr(E)$.
The portion of
$\cQ_X$ whose vertices $\lambda$ satisfy $V_\lambda \ne 0$ is called
the \textit{support} of $E$ and denoted by $\supp(E)$.
The $G$-action on $E$ determines a linear map $V_{\lambda} \to V_\mu$
for all $\lambda,\mu$ in $\supp(E)$ satisfying
$\Ext^1_X(\cO_X(\lambda),\cO_X(\mu))^G = \kk$.

Given a homogeneous bundle $E$, we denote by $[E]$ the corresponding
representation and we talk indifferently of the support of $E$ or
of $[E]$.

\begin{Lemma}
We have that
$$
\gr\left( S^n U_1 \otimes S^n U_2 \otimes \cO_X\right) =
 \bigoplus_{t,k \in \inter{0,n}} \cO_X(-n+2k,-n+2t).
$$
\end{Lemma}

\begin{proof}
First start by computing $\gr(S^n U_1 \otimes \cO_X)$ by induction,
observing that, for $n=1$, we have an $\SL_2(\kk)$-equivariant exact sequence:
\begin{equation}
  \label{P1} 
0 \rightarrow \cO_X(-1,0) \rightarrow U_1 \otimes \cO_X \rightarrow \cO_X(1,0) \rightarrow 0. 
\end{equation}
This gives rise, for any $n$, to:
$$
0 \rightarrow S^{n-1}U_1 \otimes\cO_X(-1,0) \rightarrow S^n U_1 \otimes \cO_X \rightarrow \cO_X(n,0) \rightarrow 0
$$
We get the following:
$$
\gr(S^n U_1 \otimes \cO_X) = \bigoplus_{k \in \inter{0,n}} \cO_X(-n+2k,0).
$$
Analogously, we have that
$$
\gr(S^n U_2 \otimes \cO_X) = \bigoplus_{t \in \inter{0,n}} \cO_X(0,-n+2t)
$$

The proof is achieved observing that, for any pair of $G$-homogeneous bundles $E$ and $F$, we have
 $\gr(E\otimes F) = \gr(E) \otimes \gr(F)$. 
\end{proof}

The previous lemma accounts for the vertices in $\supp(\cE)$,
which are:
\[
  (n,n) - \{2(k,t) \mid (0,0) \ne (k,t) \in \inter{0,n} \times \inter{0,n} \}.
\]

Let us look at the arrows of $[\cE]$. We start by observing that
the linear map in $\cQ_X$ arising from \eqref{P1} is
non-zero. More generally, the support
of $S^n U_1 \otimes \cO_X$ is:
$$
\xymatrix@-3ex{
  n & n-2 & n-4 & \cdots & -n+2 & -n \\
 \bullet \ar[r]  & \bullet \ar[r]  & \bullet  &  \cdots & \bullet \ar[r] & \bullet 
}
$$
The arrows correspond to elements of $\Ext^1_X(\cO_X(a+2,0),\cO_X(a,0))$.
Note that all maps in $\cQ_X$ associated with $[S^n U_1 \otimes \cO_X]$ are non-zero.
We get the following picture for $\supp(\cE)$, where all the
associated linear maps are non-zero.
\begin{equation}
  \label{quiver-n}
\xymatrix@-3ex{
&  n & n-2 & n-4 & \cdots & -n+2 & -n \\
-n & \bullet \ar[r] & \bullet \ar[r] & \bullet  & \cdots & \bullet \ar[r] & \bullet \\
-n +2 & \bullet \ar[r] \ar[u] & \bullet \ar[r] \ar[u] & \bullet \ar[u] &  \cdots & \bullet \ar[r] \ar[u] & \bullet \ar[u]\\
\vdots & \vdots & \vdots & \vdots & & \vdots & \vdots\\
n-4 & \bullet \ar[r] & \bullet \ar[r] & \bullet  & \cdots & \bullet \ar[r] & \bullet \\
n-2 & \bullet \ar[r] \ar[u] & \bullet \ar[r] \ar[u] & \bullet \ar[u] &  \cdots & \bullet \ar[r] \ar[u] & \bullet \ar[u]\\
n &  & \bullet \ar[r] \ar[u] & \bullet \ar[u] &  \cdots & \bullet \ar[r] \ar[u] & \bullet \ar[u]
}
\end{equation}

Here, the side labels denote the degrees of the summand $\cO_X(a,b)$
in the associated graded bundle. Moreover, the horizontal (resp. vertical) arrows are determined by
$\Ext^1(\cO_X(a+2,b),\cO_X(a,b))$
(resp. $\Ext^1(\cO_X(a,b+2),\cO_X(a,b))$). 
We will call \textit{main diagonal} of the support the set of
vertices of the form $(a,-a)$.

Consider a subrepresentation $[\cE']$ of $[\cE]$.
Note that all arrows of $[\cE]$ are isomorphisms and every vertex
in $\lambda \in \supp(\cE)$ is connected to another vertex to the right
of $\lambda$ or above $\lambda$ until reaching $(-n,-n)$.
Then the main observation
is that, if a
vertex $\lambda_1=(a_1,b_1)$ 
is in the support of $[\cE']$, then every vertex of $\supp(\cE)$ to right of
$\lambda_1$ or above $\lambda_1$
is also in the support of $[\cE']$, that is 
 $(a_2,b_2) \in \supp(\cE')$ if 
$-n\leq a_2 \leq a_1$ and $-n\leq b_2 \leq b_1$.
Therefore, $[\cE']$ is completely described by
means of its \textit{boundary vertices}, namely the vertices $(a,b)$
in $\supp(\cE')$, such that neither $(a+2,b)$ nor $(a,b+2)$ is in
$\supp(\cE')$.

\begin{Example}
Let us consider $n=7$ and a homogeneous bundle $\cE'$ whose
representation has the following support.
$$
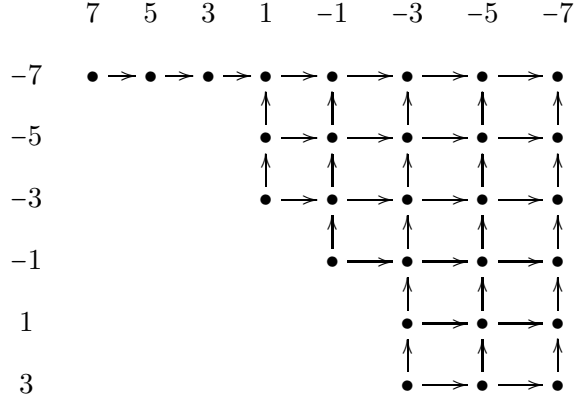
\xymatrix@-3ex{
& 7 & 5 & 3 & 1 & -1 & -3 & -5 & -7\\
-7 & \bullet \ar[r] & \bullet \ar[r] & \bullet \ar[r] & \bullet \ar[r] & \bullet \ar[r] & \bullet \ar[r] & \bullet \ar[r] & \bullet \\
-5 & &&  & \bullet \ar[r] \ar[u] & \bullet \ar[r] \ar[u] & \bullet \ar[r] \ar[u] & \bullet \ar[r] \ar[u] & \bullet  \ar[u]\\
-3 & &&  & \bullet \ar[r] \ar[u] & \bullet \ar[r] \ar[u] & \bullet \ar[r] \ar[u] & \bullet \ar[r] \ar[u] & \bullet  \ar[u]\\
-1 & &&  &   & \bullet \ar[r] \ar[u] & \bullet \ar[r] \ar[u] & \bullet \ar[r] \ar[u] & \bullet \ar[u]\\
1 & &  & &  & & \bullet \ar[r] \ar[u] & \bullet \ar[r] \ar[u] & \bullet  \ar[u]\\
3 &  && &  & & \bullet \ar[r] \ar[u] & \bullet \ar[r] \ar[u] & \bullet  \ar[u] 
}
$$
If $[\cE']$ is a subrepresentation of $[\cE]$, then all arrows must be
non-zero.
The boundary vertices are given by the four vertices of the quiver indexed by $(7,-7), (1,-3), (-1,-1)$ and $(-3,3)$.
\end{Example}

Let us introduce stability with respect to the line bundle
$\cO_X(1,1)$ in terms of representations of $\cQ_X$ according to \cite{king:moduli}.
For $\lambda=(a,b) \in \ZZ^2$, put $c_1([V_\lambda \otimes
\cO_X(\lambda)])=\rk(V_\lambda)(a+b)$ and $\rk([V_\lambda \otimes
\cO_X(\lambda)])=\rk(V_\lambda)$. For a $G$-homogeneous bundle $E$
on $X$, define $c_1([E])=c_1([\gr(E)])$ and $\rk(E)=\rk([\gr(E)])$ by linearity.
For every $G$-homogeneous bundle $E'$ we put:
$$
\mu_E([E']) = c_1([E])\rk([E']) - \rk([E])c_1([E']).
$$
The representation $[E]$ is called \textit{stable} if for all
subrepresentations $[E']$ we have that $\mu_E([E']) \geq 0$ and
the equality holds if only if $[E']$ is either $[E]$ or $[0]$.

For the $G$-homogeneous bundle $\cE$, the stability of the representation $[\cE]$ is
equivalent to the slope-stability of $\cE$ itself. Indeed,
the proof of \cite[Theorem 7.2]{ottaviani-rubei:quiver} can be adapted
to the case of $X=\PP^1 \times \PP^1$
to show that the representation 
$[\cE]$ is slope-stable if and only if $\cE \simeq W \otimes
\cE'$, where $W$ is an irreducible $G$-module and $\cE'$
is a slope-stable $G$-homogeneous bundle on $X$.
Since $\rH^1(\cE(-n,-n)) \simeq \kk$, we must then have $W \simeq
\kk$ and $\cE \simeq \cE'$.

To conclude that $\cE$ is slope-stable, we need only show that $[\cE]$
is stable, which we do in the next result.

\begin{Lemma}
For any subrepresentation $[\cE']$, we have that $\mu_{\cE}([\cE'])
\geq 0$. Moreover, $\mu_{\cE}([\cE']) = 0$ if only if either $[\cE'] = [\cE]$ or $[\cE']=[0]$.
\end{Lemma}

\begin{proof}
  Let $[\cE']$ be a non-zero subrepresentation of $[\cE]$.
  We have:
  \[
    \mu_{\cE}([\cE']) = \sum_{(a,b) \in \supp(\cE')} \left(c_1([\cE])
      - \rk([\cE])(a+b) \right)
    = \sum_{(a,b) \in \supp(\cE')} \left(-2n
      +(1-n^2)(a+b) \right).
  \]

  For a point $\lambda=(a,b) \in \ZZ^2$, write
  $\tau(\lambda)=(-b,-a)$, so that $\tau$ is the reflection along the
  main diagonal in \eqref{quiver-n}.
  Any vertex $\lambda=(a,b) \in \supp(\cE')$ satisfies $a+b = 2t$ for
  some $t \in \inter{-n,n}$. Write $\supp_t(\cE') = \{(a,b) \in
    \supp(\cE') \mid a+b=2t\}$. We get:
    \begin{equation}
      \label{halfmu}
    \frac 12 \mu_{\cE}([\cE']) =
    -n\left|\supp(\cE')\right| + \sum_{t \in \inter{-n,n}}\sum_{\lambda \in
      \supp_t(\cE')} (1-n^2)t.
    \end{equation}

  Now recall that, for any $\lambda \in \supp(\cE')$, the
  vertices of $\supp(\cE)$ to right of $\lambda$ or above $\lambda$
  are also in $\supp(\cE')$. Therefore, for any 
  $\lambda \in \supp_t(\cE')$ with $t \ge 0$, the vertex
  $\tau(\lambda)$ also lies in $\supp(\cE')$, more precisely
  $\tau(\lambda) \in \supp_{-t}(\cE')$. Note that the two terms
  $(1-n^2)t$ in the summation \eqref{halfmu} arising from a pair
  $(\lambda,\tau(\lambda)) \in \supp_t(\cE') \times \supp_{-t}(\cE')$
  add up to zero so we may
  restrict the summation to the vertices $\lambda \in \supp(\cE')$
  such that $\tau(\lambda)$ does not lie in $\supp(\cE')$. In turn
  this can happen only if $\lambda \in \supp_{-t}(\cE')$ with $t \ge 1$. Set
  $\cV_t(\cE')$ for the set of vertices $\lambda \in \supp_{-t}(\cE')$
  with $\tau(\lambda) \not \in \supp(\cE')$.
  Hence we rewrite \eqref{halfmu} as:
  \[
    \frac 12 \mu_{\cE}([\cE']) =
    -n\left|\supp(\cE')\right| -
    \sum_{t \in \inter{1,n}}\sum_{\lambda \in \cV_t(\cE')} (1-n^2)t.
  \]
  We have $\left|\supp(\cE')\right| \le n^2-1$ so:
  \begin{align*}
    \frac 12 \mu_{\cE}([\cE'])  &\ge -n(n^2-1)-
    \sum_{t \in \inter{1,n}}\sum_{\lambda \in \cV_t(\cE')}
    (1-n^2)t = \\
    & = (n^2-1)\sum_{t \in \inter{1,n}}\left(-1+\sum_{\lambda \in \cV_t(\cE')} t\right).
  \end{align*}

  Note that, since $[\cE']$ is non-zero, we must have $(-n,-n) \in \cV_n(\cE')$, hence:
  \begin{align*}
    \frac 12 \mu_{\cE}([\cE']) &\ge (n^2-1)\left(n-1+\sum_{t \in
    \inter{1,n-1}}\left(-1+\sum_{\lambda \in \cV_t(\cE')}
    t\right)\right) \ge \\
    &\ge (n^2-1)\left(n-1+(1-n)\right)=0,
  \end{align*}
where the last inequality simply follows from the fact that $t\geq 0$.

  We have thus proved that $[\cE]$ is semistable. Moreover, if
  equality is attained in the above displays, then we must have 
    $\left|\supp(\cE')\right| = n^2-1$ which implies that $[\cE']$ is
    equal to $[\cE]$.
\end{proof}

\subsection{Proof of Theorem
  \ref{main:determinant}}\label{section:provethm5} 

 This section is again in arbitrary characteristic.
All the ingredients to prove Theorem \ref{main:determinant} are ready.
According to Proposition \ref{En is stable}, the sheaf $\cE_{n-1}$ is
slope-stable in characteristic zero.
So it suffices to see that $\cT_D$ restricts over $X$ to
$\cE_{n-1}$, for a generic quadric hypersurface, isomorphic to $\PP^1
\times \PP^1$, of a linear 3-dimensional subspace of $\PP^N$.
We show next that this is indeed the case (in arbitrary characteristic).

By Lemma \ref{Lemma
i and ii} we only need to show that there is a linear space $L \simeq \PP^3
\subset \PP^N$ and a linear form $h$ over $L$ such that, in the
resulting algebra $A=A_L=R_L/I_L$, the
multiplication $\cdot h^2 : A_{n-3} \to A_{n-1}$ is an
isomorphism.

Choosing $L$ to be semigeneric in the sense of \S
\ref{subsection:semigeneric} and $h=x_0$, by Proposition \ref{propn2}
we get that $\cdot x_0^2 : A_{n-3} \to A_{n-1}$ is injective since
there is no polynomial involving $x_0^2$ in the graded piece of
degree $n-1$ of $I_L$. Moreover, we observed 
that this graded piece has dimension
$n^2$ so again $\dim(A_{n-3})=\dim(A_{n-1})={n \choose 3}$ and 
therefore $\cdot x_0^2 : A_{n-3}
\to A_{n-1}$ is an isomorphism.

Since $\cdot h^2 : A_{n-3} \to A_{n-1}$ is an
  isomorphism for a
  given choice of a linear form $h$ in $R_L$, then, by Lemma \ref{Lemma i and ii}, we get an isomorphism $\cdot g : A_{n-3}
  \to A_{n-1}$ also for a generic choice of a
  quadric form  $g$ in $R_L$.

  Next, considering $\cT_D|_L$ we get the fundamental relation:
  \[
    \rH^1(\cT_D|_L(t)) \simeq A_{t+n-1}, \qquad \mbox{for all $t \in \ZZ$},
  \]
  and these isomorphisms are compatible with the $R$-module structure.
  Then, we compute the cohomology of the restriction of $\cT_D$ to the
  quadric surface $X$ defined in $L$ by the form $g$, for $t \le 0$ by the diagram:
  \[
    \xymatrix@-2ex{
      0 \ar[r] & \rH^0(\cT_D|_X(t))   \ar[r] &
      \rH^1(\cT_D|_L(t-2)) \ar@{=}[d]  \ar^-{\cdot g}[r] &
      \rH^1(\cT_D|_L(t)) \ar@{=}[d] \ar[r] & \rH^1(\cT_D|_X(t))   \ar[r] & 0\\
&&       A_{t+n-3} \ar^-{\cdot g}[r] &        A_{t+n-1}, & 
    }
  \]
  For $t=0$ we get, by our assumption, $\rH^0(\cT_D|_X)=\rH^1(\cT_D|_X)=0$.
  It follows as in the proof of Theorem
  \ref{main:symmetric_determinant}, see the paragraph following 
  \eqref{TDP}, that $\cT_D|_X$ is isomorphic to $\cE_{n-1}$.
\medskip
  
This concludes the proof of Theorem
\ref{main:determinant}.

\section{Families of determinants}\label{section:familiesdeterminants}

Let $n \ge 3$ be an integer.
In view of Theorem \ref{main:determinant}, we know that the
logarithmic sheaf $\cT_D$ associated to the tautological determinant
$D=D_\ri$ of the $n$-matrix of 
variables is a slope-stable reflexive sheaf on $\PP^N$.
Denote by $\fM_n$ the moduli space of simple sheaves on $\PP^N$
containing $\cT_D$ or, in characteristic zero, the moduli space of
stable sheaves on $\PP^N$
containing $\cT_D$.
Our goal is to describe a dense open piece of this
moduli space as a certain group quotient.

\subsection{Moduli space and group quotient}

In view of the setup of \S \ref{intrinsic}, for
any $\rf \in \End_\kk(\AAA)$ we may consider $\det(M_\rf)$ as en
element of $S^n \AAA$. We get a rational map:
\[
  \rdet : \PP(\End_\kk(\AAA)) \dashrightarrow \PP(S^n \AAA),
\]
defined at the points where $\det(M_\rf) \ne 0$. The image of
$\rdet$ is the set of determinantal hypersurfaces of degree $n$. We denote
it by $\fD_n$.
Recall that $G=\SL(U) \times \SL(V)$ acts on $\PP(\End_\kk(\AAA))$ by
left and right composition. In terms of the matrices $M_\rf$, for
$(\bg,\bh) \in G$, the actions is
$M_{(\bg,\bh).\rf}=(\bg,\bh).M_\rf=\bh M_\rf \bg^{-1}$.
The determinant is fixed by this action, so we have a map:
\[
  \udet : \PP(\End_\kk(\AAA))/G \dashrightarrow \PP(S^n \AAA),
\]
whose image is again $\fD_n$. Recall that we put $D_\rf=\VV(\det(M_\rf))$.

\begin{Lemma} \label{stableiff}
   The following are equivalent:
  \begin{enumerate}[label=\roman*)]
  \item \label{Li} $\rf \in \GL(\AAA)$,
  \item \label{Lii} $\cT_{D_\rf}$ has no trivial direct summand,
  \item  \label{Liii} $\cT_{D_\rf}$ is simple,
  \end{enumerate}
  and, if $\cha(\kk)=0$ these conditions are equivalent to:
  \begin{enumerate}[label=\roman*), resume]
\item  \label{Liv} $\cT_{D_\rf}$ is semistable.
  \item  \label{Lv} $\cT_{D_\rf}$ is slope-stable.
  \end{enumerate}
\end{Lemma}

\begin{proof}
  In arbitrary characteristic, we have \ref{Li} $\Rightarrow$
  \ref{Liii} by Lemma \ref{simple lemma}.
  Also, \ref{Liii} implies \ref{Lii}.

  Let us check that \ref{Lii} implies \ref{Li}.
  If $\rf \in \End_\kk(\AAA) \setminus \GL(\AAA)$, then up
  to choosing a suitable basis of $\AAA$, the entries of
    the associated matrix $M_\rf$ will not form a basis of $\AAA$.
    In
    other words there is a choice of coordinates or $R$ such that not all the
    variables $(x_{i,j})_{(i,j) \in  \inter{1,n}^2}$
    occur in $M_\rf$, which is to say that $M_\rf$ is
  constant in some variable of $R$. Therefore, the equation of
  $D_\rf=\VV(\det(M_\rf))$ does not depend of this variable. 
  Hence, the partial derivative of $\det(M_\rf)$ with respect to that
  variable is zero, which in turn implies that the sheaf $\cT_{D_\rf}$
  has a trivial direct summand. We have proved \ref{Lii} $\Rightarrow$ \ref{Li}.

\medskip
In characteristic zero,
  the implication \ref{Li} $\Rightarrow$ \ref{Lv} is essentially
  Theorem \ref{main:determinant}. 
  Indeed, if $\rf \in \GL(\AAA)$ then the entries of $M_\rf$ form a basis of $\AAA$.
  Hence we can consider an appropriate change of coordinates to transform
  $M_\rf$ into the matrix of indeterminates
  $\MM_\ri=(x_{i,j})_{(i,j) \in \inter{1,n}^2}$, being $\ri$ the identity in $\GL(\AAA)$.
  This manipulation has no consequence on the stability of the
  associated sheaf and we know that $\cT_{D_\ri}$ is slope-stable, so the
  sheaf $\cT_{D_\rf}$ is slope-stable.
  Finally, $\ref{Lv}$ implies \ref{Lii}, \ref{Liii} and \ref{Liv}
  (recall Remark \ref{rmk-cone}), 
  while \ref{Lii} still implies \ref{Li}.
\end{proof}

Having this in mind, we note that, since $M_\rf$ is
canonically associated to $\rf$ and the formation of
$\cT_{D_\rf}$ is functorial,
the sheaves
$(\cT_{D_\rf} \mid [\rf] \in \PGL(\AAA))$ glue to a coherent sheaf over
$\PGL(\AAA)$ and thus yield a moduli map $\PGL(\AAA) \to \fM_n$.
This descends to a moduli map up to the action of the closed subgroup
$G=\SL(U) \times \SL(V) \subset \PGL(\AAA)$ and therefore $\Psi$
factors through the map
$\udet$. We write $\fD_n^\circ$ for the set of tautological determinantal
hypersurfaces up to a change of basis, that is, the image of
$\udet$ restricted to $\PGL(\AAA)$.
We obtain an induced map $\Phi : \fD_n^\circ  \to \fM_n$ fitting in
the following commutative diagram.
\[
\xymatrix{
\PGL(\AAA)/G  \ar[rd]_{\udet} \ar[rr]^\Psi & & \fM_n\\
& \fD_n^\circ \ar[ur]_\Phi
}
\]

\subsection{The DK-Torelli property of the determinant} \label{sec-DKTorelli-det}

We first analyze the map $\Phi$ of the above diagram via a 
Torelli-type result.
Note that Proposition \ref{prop:torelli} 
fails for $D=D_\ri$.
Indeed, $\rH^1_*(\cT_D)=0$ so of course we cannot find the elements $h_1,h_2$
required to apply Proposition \ref{prop:torelli}. Moreover,
$D$ has singularities of multiplicity $n-1$, for example the
point $(1:0:\ldots:0)$. Therefore, we couldn't use anyway
\cite{wang:jacobian} to recover $D$ from the Jacobian ideal of $D$.
In spite of this, the following result shows that $D$ enjoys the
DK-Torelli property.

\begin{Proposition} \label{Torelli-det}
  The map $\Phi$ is injective.
\end{Proposition}

\begin{proof}
  Consider the tautological determinant $D=D_\ri=\det(\MM_\ri)$.
  We give a closer look to the Gulliksen-Neg\r ard complex
considered in the proof of Proposition \ref{res-determinant}. Fixing a
basis of $U$ and $V^*$ we identify $\AAA=\Hom_\kk(U,V)$ with the vector space
$\rM_n(\kk)$ of square matrices of size $n$.
Following \cite{bruns-vetter:LNM}, we write an explicit description
of the presentation matrix of $\cT_D$, that is, of the map $\varphi$
appearing in \eqref{varphi}.
To do this, we consider a non-zero matrix $\ra \in \rM_n(\kk)$
and we describe $\varphi$ fibre-wisely over $\ra$.
Consider the complex:
$$
\kk \stackrel{\iota}{\hookrightarrow} \rM_n(\kk) \oplus \rM_n(\kk) \stackrel{\pi}{\longrightarrow} \kk
$$
with $\iota(\lambda) = (\lambda \III_n,\lambda \III_n)$ for $\lambda
\in \kk$ and $\III_n \in \rM_n(\kk)$ the
identity matrix and $\pi(\ra,\rb)= \trace(\ra-\rb)$.
The homology of this complex is a vector space of dimension $2(n^2-1)$.
The map $\varphi$ is induced at the point corresponding to a matrix
$\ra$ by:
$$
\begin{array}{rccc}
\psi_\ra: & \rM_n(\kk) & \longrightarrow & \rM_n(\kk) \oplus \rM_n(\kk)\\
& \rb & \mapsto &(\ra \rb,\rb \ra)
\end{array}
$$

Up to the choice of a new basis of $U$ and $V^*$, we may suppose
that $\ra$ is diagonal.
On the other hand, the rank of $\varphi$, and hence of $\cT_D$ at a diagonal matrix
$\ra$ can be read off from the expression of 
$\psi_\ra$. Indeed, if $\ra$ is invertible then $\ker(\psi_\ra)$ is spanned
by $\ra^{-1}$, while for $\ra$ of rank $n-k$, with $k \in \inter{1,n-1}$,
writing 
$\ra=\diag(\lambda_1,\ldots,\lambda_{n-k},0,\ldots,0)$ with $\lambda_i
\ne 0$ for all $i \in \inter{1,n-k}$ we see that $\ker(\psi_\ra)$
consists of matrices $\rb=(b_{i,j})$ with $b_{i,j}=0$ for $i \le n-k$ or
$j \le n-k$. Summing up, for $\ra$ of rank $n-k$, with $k \in
\inter{1,n-1}$, we have
\[
  \rk(\cT_D|_\ra) = n^2+k^2-2.
\]

This gives:
$$
\left\{\ra \in \PP^N \mid \rk(\ra) =1 \right\} = \left\{\ra \in \PP^N \mid
  \rk(\cT_D|_\ra) =2n^2-2n-1 \right\}. 
$$
In other words, the locus of rank-1 matrices is the support of the
Fitting ideal of $\cT_D$ defined by the minors of order $2n$ of
$\varphi$. 

Now, 
the hypersurface $D$ is determined as the variety of $(n-1)$-secant
subspaces of dimension $n-2$ to the
locus of matrices of rank $1$. This says in particular that $\cT_D$
determines $D$ as the $(n-1)$-secant variety to the locus where
$\cT_D$ has rank $2n^2-2n-1$.

After an appropriate change of coordinates, as mentioned before, we
get that for every hypersurface  $D_\rf \in \fD_n^\circ$, the
associated reflexive sheaf $\cT_{D_\rf}$ determines $D_\rf$.
\end{proof}

\subsection{The determinant as a $2:1$ cover}

Here we show that the fibre of the map
$$\udet : \PGL(\AAA)/\SL(U)\times \SL(V) \to
\fD_n^\circ$$ consists of $2$ distinct points.
 This result goes back to Frobenius,
  \cite[\S 7.1]{frobenius:gruppen}.
It has been extended and recasted in various ways, let us mention
\cite[\S 4]{waterhouse:automorphisms-basic}, \cite[\S
8]{bermudez-garibaldi-larsen}, see also
\cite{dieudonne:orthogonal-4,marcus-moyls:linear}
We provide a proof for self-containedness and to point out a slightly
different approach based on divisor class groups involving Ulrich
sheaves, close to the methods of \cite{waterhouse:automorphisms-basic,reichstein-vistoli:determinantal}.

\begin{Proposition} \label{2:1}
  The morphism $\udet$ is set-theoretically $2:1$.  
\end{Proposition}

\begin{proof}
 By the argument of Proposition
  \ref{Torelli-det}, it is enough to prove that the set-theoretic
  fibre of $\udet$ at $D=D_\ri=\det(M_\ri)$ consists of two distinct points.
  To do this, we look more closely at the geometry of a resolution of
  singularities $\sigma^+ : D^+ \to D$ and argue that, up to the $G$-action, the
  elements $\rf \in \PGL(\AAA)$ such that $D_\rf=D$ are in bijection
  with effective divisor classes $\fl$ on $D^+_\rf$ such that:
  \[
    \fl \cdot \fh^{n^2-3}={n \choose 2},
  \]
  where $\fh$ is the pull-back to $D^+$ of the hyperplane class of
  $D \subset \PP^N$. We then show that there are precisely two such
  divisor classes.

  To define $D^+$, we consider $\PP(V)$ and the tautological quotient
  bundle $\cQ_+$ of rank $(n-1)$ on $\PP(V)$, defined by the Euler sequence:
  \[
    0 \to \cO_{\PP(V)}(-1) \to V^* \otimes \cO_{\PP(V)} \to \cQ_+ \to 0.
  \]

  Put $D^+=\PP(U \otimes \cQ_+)$. Note that $\rH^0(U \otimes \cQ_+) \simeq
  \AAA$.
  Geometrically, we have:
  \[
    D^+=\{([v],[\ra]) \in \PP(V) \times \PP(\AAA) \mid v \circ \ra =0 \}.
  \]
  The linear system associated with
  the tautological relatively ample divisor $\fh$ defines a birational
  morphism $\sigma^+ : D^+ \to D$. 
  Denote by $\fl^+$ the pull-back to $D^+$ of a hyperplane of
  $\PP(V)$ via the bundle map $\pi^+ : D^+ \to \PP(V)$.
   The map $\sigma^+$ is an isomorphism away from the singular locus $\sing(D)$
   of $D$ which consists of the matrices $\ra : U \to V$ of rank at
   most $n-2$. This locus has codimension $3$ in $D$.
   The maps $\pi^+$ and $\sigma^+$ are the restrictions to $D^+$
   of the projections from $\PP(V) \times \PP(\AAA)$ onto the first and
   second factor. Note that the generic fibre of $\sigma^+$ over $\sing(D)$
   is a projective line, so the exceptional locus $\fe^+$ of
   $\sigma^+$ has codimension $2$ in $D^+$. Therefore $\sigma^+$
   induces an isomorphism:
   \[
     \Cl(D) \simeq \Pic(D^+) \simeq \ZZ \fh \oplus \ZZ \fl^+.
   \]
  
  Note that $D^+$ is cut in $\PP(V) \times \PP(\AAA)$ by a linear
  section, whose Koszul complex reads:
  \[
    0 \to \cO_{\PP(V) \times \PP(\AAA)}(-n,-n) \to
    \cdots \to U \otimes \cO_{\PP(V) \times \PP(\AAA)}(-1,-1) \to
    \cO_{\PP(V) \times \PP(\AAA)} \to \cO_{D^+} \to 0.
  \]

  Set ${}^\tra \fl^+ = (n-1)\fh - \fl^+$.
  From the above complex we compute:
  \[
    \rH^0(\cO_{D^+}({}^\tra \fl^+)) \simeq U^*.
  \]
  To see this, for $i \in \inter{1,n}$, set $\cK_j$ for the image of
  the $j$-th differential of the Koszul complex, taking the form:
  \[
    \bigwedge^j U \otimes \cO_{\PP(V) \times \PP(\AAA)}(-j,-j) \to
    \bigwedge^{j-1}U \otimes \cO_{\PP(V) \times \PP(\AAA)}(1-j,1-j).
  \]
  For all $j \in \inter{1,n}$, the Künneth formula gives
  $\rH^p(\cK_j(-1,n-1))=0$ for $p\in \NN \setminus \{j\}$. We obtain:
  \[
    \rH^0(\cO_{D^+}({}^\tra \fl^+)) \simeq  \rH^1(\cK_1(-1,n-1))
    \simeq \cdots \simeq \rH^{n-1}(\cK_{n-1}(-1,n-1))
    \simeq \bigwedge^{n-1} U \simeq U^*.
  \]

  The linear system $|{}^\tra \fl^+|$ gives a rational
  map $D^+ \dashrightarrow \PP(U^*)$.
  Resolving the indeterminacies of this map we get a
  variety $\hat D$ and a morphism $\hat D \to
  \PP(U^*)$. Geometrically:
  \[
    \hat D=\{([v],[\ra],[u]) \in \PP(V) \times \PP(\AAA) \times
    \PP(U^*) \mid v \circ \ra =0=\ra \circ u \}.
  \]

  Starting from $\PP(U^*)$ and the quotient bundle $\cQ_-$ over
  $\PP(U^*)$ we get second a desingularization $D^-=\PP(V^* \otimes
  \cQ_-)$ with a birational map $\sigma^- : D^- \to D$. This is
  described as:
  \[
    D^-=\{([\ra],[u]) \in \PP(\AAA) \times \PP(U^*)\mid \ra  \circ u =0 \}.
  \]
  The manifold $\hat D$ is the blow-up of $D$ along $\sing(D)$ and the map
  $\hat D \to D$ is a $\PP^1 \times \PP^1$-bundle over the smooth
  locus of $\sing(D)$.

  We look at the effective cone of $D^+$. Tensoring the Koszul complex
  above with $\cO_{\PP(V) \times \PP(\AAA)}(x,y)$, for some $(x,y) \in
  \ZZ^2$, we see $\rH^0(\cO_{D^+}(x\fl^+ + y \fh))=0$ if $y < (1-n)x$
  or if $y<0$. So the effective cone of $D^+$ is spanned
  over $\QQ$ by $\fl^+$ and ${}^\tra \fl^+$. Therefore, an effective
  divisor on $D^+$ is of the form $x\fl^+ + y \fh$, with:
  \begin{equation}
    \label{condizioni}
   (x,y) \in \ZZ \times \NN, \qquad
    \mbox{and:} \qquad
    y \ge (1-n)x.
  \end{equation}
  We compute:
  \[
    \fl^+ \cdot \fh^{n-3}=  {}^\tra\fl^+ \cdot \fh^{n-3}={n \choose
      2}, \qquad \fh^{n-2}=n.
  \]

  Choose now $[\rf] \in \PGL(\AAA)$ such that the determinant $D_\rf$ of
  the matrix $M_\rf$  satisfies $D_\rf=D$.
  Then the coherent sheaf $\cL_\rf=\coker(M_\rf)$ is a rank-one reflexive
  sheaf over $D$, 
  actually an Ulrich sheaf. Similarly we get a second Ulrich sheaf of
  rank $1$ as:
  \[
    {}^\tra \cL_\rf = \coker\left({}^\tra M_\rf : V^*\otimes
      \cO_{\PP(\AAA)}(-1) \to U^*\otimes \cO_{\PP(\AAA)}\right);
    \qquad \mbox{we have:} \qquad     \cL_\rf(1-n)^* \simeq {}^\tra \cL_\rf.
  \]
  Each of these sheaves determines an element of $\End(\AAA)$ up to
  $G$-action arising as the minimal presentation matrix of the module
  of global sections of the sheaf, in some basis.
  
  Next, note that a non-zero global section of $\cL=\cL_\rf$ vanishes along the Weil
  divisor $B$ of $D$ consisting of matrices of size $n\times (n-1)$ that
  have rank at most $n-2$. This pulls-back via $\sigma^+$ to an
  effective divisor of $D^+$ 
  of the form $x\fl^+ + y \fh$, for some $(x,y) \in \ZZ \times
  \NN$. The degree of $B$ 
  in $\PP(\AAA)$ is ${n \choose 2}$ so:
  \[
    (x\fl^+ + y \fh)\cdot \fh^{n-2}={n \choose 2}, \qquad \mbox{so:}
    \qquad y=\frac{(n-1)(1-x)}{2}.
  \]
  
Together with \eqref{condizioni}, this gives two possibilities for
$(x,y)$, namely either $(x,y)=(1,0)$, in which
case the divisor class is $\fl^+$, or $(x,y)=(-1,n-1)$ so that the divisor
class is ${}^\tra\fl^+$.

In turn, $|\fl^+|$ gives the rational projection $D \dashrightarrow
\PP(V)$ and the sheaf $\cL=\sigma^+_*(\cO_{D^+}(\fl^+))$,
while $|{}^\tra \fl^+|=|\fl^-|$ gives $D \dashrightarrow \PP(U^*)$,
and the sheaf ${}^\tra \cL=\sigma^-_*(\cO_{D^-}(\fl^-))$,
the indeterminacies of these maps being resolved by
$\pi^+:D \dashrightarrow \PP(V)$
and $\pi^-:D \dashrightarrow \PP(U^*)$ and simultaneously over $\hat
D$.
The two possible divisors of degree ${n \choose 2}$ give 
precisely two points in the fibre of $\udet$ over $D$.
\end{proof}

\subsection{The Hilbert scheme}

Our next goal is to prove that $\Psi$ is a local isomorphism.
To do this, we consider a further space, which we
denote by $\fH_n$.
This is defined as the Hilbert scheme of subschemes of $\PP(\AAA)$
having the same Hilbert polynomial as $\sing(D)$.
Given any $\rf \in \PGL(\AAA)$, the minors of order $(n-1)$ of the
matrix $M_\rf$ cut a subscheme lying in $\fH_n$, so the assignment $[\rf] \mapsto
Z_\rf = \sing(D_\rf)$ defines a morphism:
$$\Xi : \PGL(\AAA) \to \fH_n,$$
whose image we denote by $\fH_n^\circ$.

Given $Z = \sing(D_\rf) \in \fH_n^\circ$, the ideal homogeneous ideal
$I_Z$ is minimally generated by the $n^2$ minors of degree $n-1$ and
the kernel of this set of generators (namely the first syzygy) determines the module $T_D$ up to
isomorphism. After sheafification, this yields a morphism $\rsyz :
\fH_n^\circ \to \fM_n$. Summing up, we get a different factorization of $\Psi$ as:

\[
\xymatrix{
\PGL(\AAA)/G  \ar[rd]_{\Xi} \ar[rr]^\Psi & & \fM_n\\
& \fH_n^\circ \ar[ur]_{\rsyz}
}
\]

\begin{Proposition} \label{Psi}
  The morphisms $\rsyz$ and $\Psi$ are submersions.
\end{Proposition}

\begin{proof}
Up to an appropriate change of coordinates, it is enough to prove
the statement at the point $\ri$ associated to the tautological matrix
of indeterminates $(x_{i,j})_{1 \le i,j \le n}$, as mentioned before.
We consider the sequence \eqref{syz} and denote by $\xi \in
\Ext^1_{\PP(\AAA)}(\cI_Z(n-1),\cT_D)$ the extension corresponding to that sequence.
We proved in Lemma \ref{simple lemma}
that $\Ext^1_{\PP(\AAA)}(\cI_Z(n-1),\cO_{\PP(\AAA)})=0$ and
 $\Hom_{\PP(\AAA)}(\cI_Z(n-1),\cO_{\PP(\AAA)})=0$ so
$\Hom_{\PP(\AAA)}(\cT_D(n-1),\cO_{\PP(\AAA)})$ is
canonically identified with $\AAA^*$ and hence
$\cI_Z(n-1)$ is recovered as cokernel of the
dual evaluation:
\[
  \cT_D \to \Hom_{\PP(\AAA)}(\cT_D(n-1),\cO_{\PP(\AAA)})^* \otimes \cO_{\PP(\AAA)}.
\]

The upshot is that the map $\rsyz$ is injective, for the ideal sheaf
of the subscheme $Z \subset \PP(\AAA)$ is reconstructed by $\cT_D$.
Therefore $\rsyz$ is bijective as it is surjective by definition.

\medskip

Taking $\Hom_{\PP(\AAA)}(\cI_Z(n-1),-)$ of \eqref{syz} and using
\eqref{vanishingI} gives a natural isomorphism:
\[
  \wedge \xi : \Ext^1_{\PP(\AAA)}(\cI_Z,\cI_Z) \to \Ext^2_{\PP(\AAA)}(\cI_Z(n-1),\cT_D).
\]

Next we observe that, applying $\Hom_{\PP(\AAA)}(-,\cT_D)$ to
\eqref{syz} and using \eqref{vanishingII} we get a natural isomorphism:
\[
  \wedge \xi :  \Ext^1_{\PP(\AAA)}(\cT_D,\cT_D) \to \Ext^2_{\PP(\AAA)}(\cI_Z(n-1),\cT_D).
\]
We get an isomorphism $\Ext^1_{\PP(\AAA)}(\cI_Z,\cI_Z) \to
\Ext^1_{\PP(\AAA)}(\cT_D,\cT_D)$ induced by $\wedge \xi$ and $(\wedge \xi)^{(-1)}$
which corresponds to the differential of $\rsyz$ at the point $Z$ of
$\fH_n^\circ$. We have showed that $\rsyz$ is a submersion.
Finally, we use that $\Xi$ is a submersion, see
\cite[Corollary 7.6]{kleppe-miro2020:deformation}.
 This can also be obtained by explicit computation of the
  differential as in \cite[Lemma 4.1]{reichstein-vistoli:determinantal}.
Summing up, we proved that $\Psi$ is also a submersion.
\end{proof}

\subsection{Proof of Theorem \ref{main:determinant-moduli}}

In order to prove Theorem \ref{main:determinant-moduli}, we show the
following more precise result.

\begin{Theorem}
  The map $\Phi$ is an open immersion onto a smooth affine piece of an
  irreducible component of $\fM_n$ of dimension $(n^2-1)^2$. The map
  $\udet$ is an étale $2:1$ cover onto $\fD_n^\circ$.
\end{Theorem}

\begin{proof}
In order to prove this, in view of Proposition \ref{Torelli-det} and
Proposition \ref{2:1} it suffices to show that $\Psi$ is a
submersion and that the image of $\Psi$ is affine.
The first statement is proved in Proposition \ref{Psi}. The fact that
the image of $\Psi$ is affine follows from the fact that 
$\PGL(\AAA)/G$ is affine, as $\PGL(\AAA)$ is affine
and $\PGL(\AAA)/G$ is the spectrum of the ring of
$G$-invariants of the coordinate ring of $\PGL(\AAA)$,
see \cite[Chapter IV]{humphreys:linear_algebraic_groups}.
So the image of $\Psi$ is affine as well as it is the quotient of
$\PGL(\AAA)/G$ by the free $\ZZ/2\ZZ$-action given by transposition.
\end{proof}



\def\cprime{$'$} \def\cprime{$'$} \def\cprime{$'$} \def\cprime{$'$}
  \def\cprime{$'$} \def\cprime{$'$} \def\cprime{$'$} \def\cprime{$'$}
  \def\cprime{$'$}
\providecommand{\bysame}{\leavevmode\hbox to3em{\hrulefill}\thinspace}
\providecommand{\MR}{\relax\ifhmode\unskip\space\fi MR }
\providecommand{\MRhref}[2]{%
  \href{http://www.ams.org/mathscinet-getitem?mr=#1}{#2}
}
\providecommand{\href}[2]{#2}

\end{document}